\documentclass[reqno,centertags,cmex10]{amsart}
\usepackage{amsfonts,amssymb,curves,latexsym,hyperref}
\usepackage[dvipsnames,svgnames,x11names,hyperref]{xcolor}
\hypersetup{colorlinks,breaklinks,
             urlcolor=blue,
             linkcolor=blue}
\usepackage{tikz-cd}
\usepackage[normalem]{ulem}
\newtheorem{theorem}{Theorem}
\newtheorem{proposition}[theorem]{Proposition}
\newtheorem{lemma}[theorem]{Lemma}
\newtheorem{example}[theorem]{Example}
\newtheorem{remark}[theorem]{Remark}
\newtheorem{corollary}[theorem]{Corollary}
\newtheorem{definition}[theorem]{Definition}

\newcommand{\C}{\mathcal{C}}

\def\E{\mathcal{E}}

\def\Fq{{\mathbb F}_q}
\def\A{{\mathbb A}}
\def\D{{\mathcal D}}
\def\Fqn{{\mathbb{F}}_q^n}
\def\Fqm{{\mathbb F}_{q^m}}

\def\Fqmn{{\mathbb F}_{q^m}^n}
\def\rank{\mathop{\rm rank}}
\def\M{{\mathbb M}}
\def\MM{\mathcal{M}}
\def\MC{{\mathcal{M}_{\C}}}
\def\Mat{{\mathbb M}_{m\times n}(\Fq)}
\def\F{\mathbb{F}}

\def\FQ{\mathbb{F}_{Q}}
\def\FQt{\mathbb{F}_{\tilde{Q}}}

\def\nn{\textbf{n}}

\def\r{\mathbf{r}}

\begin{document}

\title{Weight spectra of Gabidulin rank-metric codes and Betti numbers}

%matrix weights of Delsarte codes}

%\author[Johnsen]{Trygve Johnsen}
%\address{Department of Mathematics and Statistics, %  %\newline \indent
% UiT-The Arctic University of Norway  \newline \indent %Troms{\o}, 
%N-9037 Troms{\o}, Norway}
%\email{trygve.johnsen@uit.no}
\thanks{The authors are partially supported by grant 280731 from the Research Council of Norway,  and  by the project ``Pure Mathematics in Norway" through the Trond Mohn Foundation and Tromsø Research Foundation.} %\email{Trygve.Johnsen@uit.no} \and \email{Hugues.Verdure@uit.no} }.}
%\thanks{\leftline{and T. Johnsen was partially supported by RCN, Norway, both authors as a part of the joint}
%\leftline{Indian-Norwegian project "Mathematical Aspects of Information Transmission: Effective Error} 
%\leftline{Correcting Codes".} }
\author[Johnsen]{Trygve Johnsen}
\address{Department of Mathematics and Statistics, %  \newline \indent
 UiT-The Arctic University of Norway  \newline \indent %Troms{\o}, 
N-9037 Troms{\o}, Norway}
\email{trygve.johnsen@uit.no}
%\thanks{Hugues Verdure is partially supported by grant 280731 from the Research Council of Norway.}

\author[Pratihar]{Rakhi Pratihar}
\address{Department of Mathematics and Statistics, %  \newline \indent
 UiT-The Arctic University of Norway  \newline \indent %Troms{\o}, 
N-9037 Troms{\o}, Norway}
\email{rakhi.pratihar@uit.no}
%\thanks{Rakhi Pratihar is supported by grant 280731 from the Research Council of Norway.}

\author[Verdure]{Hugues Verdure}
\address{Department of Mathematics and Statistics, %  \newline \indent
 UiT-The Arctic University of Norway  \newline \indent %Troms{\o}, 
N-9037 Troms{\o}, Norway}
\email{hugues.verdure@uit.no}
%\thanks{Hugues Verdure is partially supported by grant 280731 from the Research Council of Norway.}

\date{\today}

\begin{abstract}
 We consider $q$-matroids and their associated classical matroids derived from Gabidulin rank-metric codes. We express the generalized rank weights of a Gabidulin rank-metric code in terms of Betti numbers of the dual classical matroid associated to the $q$-matroid corresponding to the code. In our main result, we show how these Betti numbers and their elongations determine the generalized weight polynomials for $q$-matorids, in particular, for the Gabidulin rank-metric codes. In addition, we demonstrate how the weight distribution and higher weight spectra of such codes can be determined directly from the associated $q$-matroids by using M\"{o}bius functions of its lattice of $q$-flats.
\end{abstract}

\maketitle

\section{Introduction} \label{intro}

 Rank-metric codes were first introduced by Delsarte \cite{Del} in 1978, and are referred to as \textit{Delsarte rank-metric codes}. These codes are defined as $\Fq$-subspaces of the space of $m \times n$-matrices over $\Fq$ (the field with $q$ elements, for $q$ a prime power) where the \textit{rank distance} between two codewords (i.e. matrices) is given by the rank of their difference. These codes are important for their applications in network coding, public-key cryptography, and distributed storage, topics that thus stimulate the interest in studying the intrinsic properties of these codes. In this work we are interested in a particular case of Delsarte rank-metric codes, widely known as \textit{Gabidulin rank-metric codes}, introduced independently by Gabidulin \cite{Gab85} in 1985 and Roth \cite{Rot91} in 1991. A \textit{Gabidulin rank-metric code} $\C$ of length $n$ and dimension $k$ may be defined as a $k$-dimensional subspace of the $n$-dimensional vector space $\Fqm^n$ over $\Fqm$. In this case, the rank-analogue of the Hamming metric is defined as follows. Fix a basis $B$ of $\F_{q^m}$ over $\Fq$ to interpret a vector $\underline{v} \in \Fqm^n$ as an $(m \times n)$-matrix $\A_{\underline{v}}$ over $\Fq$. Now the rank distance between two codewords $\underline{v}$ and $\underline{w}$ is defined as $d(\underline{v}, \underline{w}):= \rank(\A_{\underline{v}}-\A_{\underline{w}})$. In this paper we also consider the Gabidulin rank-metric code which is the set $\C_r$ of all $\FQt$-linear combination of words of $\C$, where $\tilde{Q}=Q^{r}$ for an arbitrary but fixed positive integer $r$, equipped with the rank metric. This code is referred to as the extension of $\C$ to $\FQt$ and denoted by $\C \otimes_{\FQ} \FQt$ (or $\tilde{\C}$, in short) analogous to the case of Hamming metric codes as considered in \cite{JRV16,JP13}. 
 
 For a linear block code $C$ over $\Fq$ and its extension $C \otimes_{\Fq} \FQ$ for some $Q = q^m$ with any positive integer $m$, it is shown in \cite{JP13} that the number $A_{C,j}(Q)$ of codewords of Hamming weight $j$ in the code $C \otimes_{\Fq} \FQ$ is a polynomial in $Q$. Later, in \cite{JRV16}, one generalized the weight polynomials to matroids. Following the work on the connection between generalized weights of a Hamming metric code $C$ and Betti numbers of certain associated matroid $M$ as established in \cite{JV2}, an expression for $A_{C,j}(Q)$ or more generally, for the generalized weight polymials $P_{M,j}(Q)$ of a matroid $M$, is provided in terms of Betti numbers associated to the Stanley-Reisner ring of the matroid $M$ and its elongations. On the other hand, recently a study on determining the singular homology of $q$-complexes associated to $q$-matroids has been initiated in \cite{GPR21}. This work is towards a topological approach to connect the generalized rank weights of a rank metric code with homological invariants of the associated $q$-matroid. This led us to study the $q$-matroids associated to Gabidulin rank metric codes and their lattices of cycles and flats with a view towards a combinatorial approach to express generalized rank weights and related polynomials in terms of homology of the associated $q$-matroids.
 
In this paper, we prove rank-analogues of some classical results for Gabidulin rank-metric codes. To each Gabidulin rank-metric code $\C$, we associate $q$-matroids $\MC$ and $\MM_{\C}^*$ as introduced in \cite{JP}; which can be viewed as $q$-analogues of matroids derived from generator matrices and parity check matrices, respectively, for block codes with the Hamming metric. To $\MC$ we furthermore associate a classical matroid $Cl(\MC)$ (also mentioned in \cite{JP}). We study resolutions of the Stanley-Reisner ring of the independence complexes of the dual classical matroid $Cl(\MC)^*$, and the possible elongations of that matroid. We show how the rank-weight distribution of the code $\C \otimes_{\FQ} \FQt$ can be determined by the $\mathbb{N}$-graded Betti numbers of the Stanley-Reisner ring corresponding to these resolutions. Furthermore, we generalize the results for arbitrary $q$-matroids by introducing generalized weight polynomials $P_{\MM,j}$ for a $q$-matroid $\MM$. 

%The higher weight spectra $A^{(r)}_j$, for all $r$ and fixed, but arbitrary $j$ on one hand, and $P_j$ on the other, determine each other (the proof will essentially) go as for block codes. Furthermore we will show how the knowledge of the $P_j$ will determine, not necessarily the individual Betti numbers, but $\phi^l_{j}$ of these Betti numbers. Moreover one can actually  show that the knowledge only of all the $\phi^l_{j}$, and not necessarily each individual Betti-number, is sufficient to retrieve the $P_j$ (and of course then also the $A^{(r)}_j$). 

Moreover, we show that the following three pieces of information are equivalent: The weight distribution $A_{\C,j}(Q^r)$ for all $j,r$, the higher weight spectra $A^{(i)}_{Q,j}$ of $\C$ for all $i,j$, and all the $\phi^{(l)}_{j}$ (certain alternating functions for $Cl(\MC)^*$ for all $i,l$). This result is a perfect analogue of a corresponding result for usual Hamming block codes and its associated matroids, see \cite[Cor. 17]{JV20}. In particular we like to emphasize how the higher rank-weight spectra (for subcodes of all dimensions) of $\C$ follow from the rank-weight distribution of $\tilde{C}$.

%Furthermore we will show how the knowledge of the $P_j$ will determine, not necessarily the individual Betti numbers, but certain alternating functions $\phi^l_{j}$ of these Betti numbers. Moreover one can actually  show that the knowledge only of all the $\phi^l_{j}$, and not necessarily each individual Betti-number, is sufficient to retrieve the $P_j$ (and of course then also the $A^{(r)}_j$). 
%A goal of the paper is to study the connection between various pieces of information associated to the codes, and how one concretely can pass from one piece to another:
%Hence the metric is given by the rank of differences between $(mr \times n)$-matrices.

%In this situation we will observe that the weight distribution $A_j(\C_r), j=0,\cdots,n$ is given as $P_j(\tilde{Q})$, for polynomials $P_j$ of degree at most $k=\dim_{\FQ}\C,$ and it will then be fairly straightforward to prove that the higher weight spectra $A^{(r)}_j$, for all $r$ and fixed, but arbitrary $j$ on one hand, and  $P_j$ on the other, determine each other (the proof will essentially) go as for block codes. \sout{These observations and facts can also be obtained, or are already obtained,from known results about the weight distributions of Delsarte rank-metric codes in general (see for example  \cite{CGLR} and \cite{BCR}), but we choose to present the material in our own way, to link it clearly to other pieces of information.} 

%rank-weight spectrum for single codewords (for $\C$, and  and also for $\C_r$) , since certain combinatorial invariants that we describe below, will determine this ``usual" weight spectrum.

A key observation that plays an instrumental role throughout this article is the isomorphism between the lattice of $q$-flats of the $q$-matroid $\MC$ (resp. lattice of $q$-cycles of $\MM_{\C}^*$) and the lattice of flats of the matroid $Cl(\MC)$ (resp. lattice of cycles of $Cl(\MC)^*$). Moreover, it is well known (see for example \cite[p.57]{S} and \cite[p. 6]{JV21}) that all the Betti numbers of the Stanley-Reisner rings of $Cl(\MC)^*$ and its elongations can be given by concrete M{\"o}bius functions of the lattice of flats of $Cl(\MC)$ and its various truncations and sublattices. We use this fact to show that sometimes it is more convenient to work directly with the lattice of $q$-flats of $\MC$ or the opposite lattice, i.e., of $q$-cycles of $\MM_{\C}^*$. 

It is important to mention that alternative methods for determining weight distribution and higher weight spectra is given in \cite{CGLR,BCR} for the more general notion of Delsarte rank-metric codes. Thus we get a variation of ways to retrieve the triple set of information described above for Gabidulin rank-metric codes; one from resolutions of Stanley-Reisner rings, and one from a direct study of lattices in addition to techniques described by the authors in \cite{CGLR,BCR} using zeta functions.
%, where a related technique, M{\"o}bius inversion of certain functions, is central).

The paper is organized as follows. In the next section, we collect some preliminaries and recall basic definitions regarding notions like ($q$-)matroids, $(q,m)$-polymatroids, and Stanley-Reisner rings associated to matroids.  
In Section \ref{sec3} we show how the generalized weights of Gabidulin rank-metric codes, more generally for $q$-matroids, can be expressed by invariants derived from the mentioned Stanley-Reisner rings. We consider the extended codes $\C \otimes_{\FQ} \FQt$ in Section \ref{sec4} and give expressions for their rank-weight distributions using the classical matroids associated to the $q$-matroids corresponding to the codes. Our main results (Theorems \ref{main} and \ref{equivalence} and Corollary \ref{follows}) determining generalized rank weights and higher weight spectra of $\C \otimes_{\FQ} \FQt$, and more generally, the weight polynomials for $q$-matroids in terms of Betti numbers, are derived in Section \ref{ncwq}. In Section \ref{examples} we demonstrate our main results with an example and also show how the weight distribution of uniform $q$-matroids can be determined directly from its lattice of $q$-flats. We end this article with a retrospective look in Section \ref{mbs} where we define a new concept called virtual Betti numbers for lattices satisfying the Jordan-Dedekind property.

%{ \textbf{Notation}: Throughout we will use the notation $[n]$ for the set $\{1, \ldots, n\}$. For clarity we throughout use the terminology \textit{classical matroid} for the standard version of matroids which are based on finite ground set.}

\section{ Preliminaries} \label{defs}
Throughout this paper $q$ denotes a power of a prime number and $\Fq$ the finite
field with $q$ elements. We fix a positive integer $n$, and denote by $E$ the set $\{1,\ldots,n\}$ and by $\mathcal{E}$ the $n$-dimensional vector space $\Fqn$ over $\Fq$. By $\Sigma(\E)$ we denote the set of all subspaces of $\E$.
 We use $[j]$ to denote the quantity $ q^{n-1}+q^{n-2}+\cdots+q^{n-j}$
for any $j=1,\cdots,n$. We denote by $\mathbb{N}_0$ the set of all nonnegative integers, and by $\mathbb{N}$ the set of all positive integers.

 In this section, we recall basic definitions and results concerning matroids, their associated Stanley-Reisner rings, and their $q$-analogues, i.e., $(q,m)$-polymatroids and $q$-matroids.
\subsection{Matroids}
 There are many cryptomorphic definitions of a matroid \cite{Whi86}.  
%We refer to~\cite{Oxley11} for a deeper study of the theory of matroids. 
Here we record the one in terms of rank function:

\begin{definition} \label{def:matroid} 
A matroid is a pair $(E,\r)$ where $E$ is a finite set and $\r: 2^E \rightarrow \mathbb{N}_0$ is a function satisfying: \begin{list}{}{\leftmargin=1em\topsep=1.5mm\itemsep=1mm}
\item[{\rm (R1)}] If $X \subseteq E$, then $0 \leqslant \r(X) \leqslant|X|$,
\item[{\rm (R2)}] If $X \subseteq Y \subseteq E$, then $\r(X) \leqslant \r(Y)$,
\item[{\rm (R3)}] If $X,Y$ are subsets of $E$, then \[\r(X \cap Y) + \r(X \cup Y) \leqslant \r(X)+\r(Y).\]
\end{list}
\end{definition}

The function $\r$ is called the rank function of the matroid. The rank of a matroid $M=(E,\r)$ is $\r(E)$. The nullity function $\nn$ of the matroid is given by $\nn(X) = |X|-\r(X)$ for $X \subseteq E$. By (R1), this is an integer-valued non-negative function on $2^E$.

 Next we recall the notions of flats and cycles of a matroid, which play a central role in this paper.

\begin{definition}\label{def:flat}
Let $M=(E,\r)$ be a matroid. A flat of the matroid is a subset $F \subseteq E$ satisfying \[ \ \r(F \cup \{x\}) = \r(F) +1, \text{ for all } x \in E\backslash F.\]
\end{definition}

From the definition, it follows easily that $E$ is a flat itself. Also, for $X \subseteq E$, the smallest flat containing $X$ is the set $Y=\{x \in E,\ \r(X \cup \{x\}) = \r(X)\}$ and moreover, $\r(Y)=\r(X)$. The intersection of two flats is a flat.

\begin{definition}\label{def:cycle} 
Let $M=(E,\r)$ be a matroid, and $\nn$ be its nullity function. For $0 \leqslant i \leqslant \nn(E)$, let \[\mathcal{N}_i =\{X \subseteq E,\ \nn(X)=i\}\] and let $N_i$ be the set of minimal elements of $\mathcal{N}_i$ with respect to inclusion. Then the elements of $N_i$ are called cycles of $M$ of nullity $i$. Cycles of nullity $1$ are called circuits.
\end{definition}

It is proved in \cite[Section 3]{JV2} that the cycles of a matroid are union of circuits, and of course, by definition, $\emptyset$ is a cycle (of nullity $0$).

\begin{definition}\label{def:dual_matroid}
Let $M=(E,\r)$ be a matroid. The dual matroid of $M$ is the matroid $M^*=(E,\r^*)$ with \[\r^*(X) := |X|+\r(E\setminus X)-\r(E).\]
\end{definition}

It is well known that there is a one to one correspondence between the flats of a matroid and the cycles of its dual, namely, $F$ is a flat of $M$ if and only if $E \setminus F$ is a cycle of $M^*$.

\begin{definition}\label{def:trunc} Let $M=(E,\r)$ be a matroid. The $l^{{\text{th}}}$ truncation of $M$ is the matroid $M_{(l)}=(E,\r_{(l)})$ where \[\r_{(l)}(X):= \min\{\r(M)-l, \r(X)\},\] 
for any subset $X \subseteq E$ and $0 \leq l \leq \r(M),$
and the $l^{{\text{th}}}$ elongation of $M$ is the matroid $M^{(l)}=(E,\r^{(l)})$, where 
 \[
 \r^{(l)}(X) := \min\{|X|, \r(X)+ l\}
  \] for any subset $X \subseteq E$ and $0 \leq l \leq n- \r(M)$. 
 \end{definition}

Note that $(M^*)^{(l)}=(M_{(l)})^*$, and 
$(M^*)_{(l)}=(M^{(l)})^*$, and that the flats of $M_{(l)}$ are exactly the flats of $M$ except those of rank $\r(M)-1,\cdots,\r(M)-l$, and that the cycles of $M^{(l)}$ are those of $M$ except those of nullity $1,2,\cdots, l$. { }

\subsection{Stanley-Reisner resolutions} \label{Stanley-Reisner}
Any matroid $M=(E,\r)$ gives rise to a simplicial complex $\Delta_M$ of independent sets, i.e., the faces of the complex are given by \[\mathcal{F} = \{ X \subseteq E,\ \r(X)=|X|\}.\] If $\mathbb{K}$ is a field, we can associate to the underlying simplicial complex a monomial ideal $I_M \subseteq S=\mathbb{K}[X_e, e \in E]$ defined by \[I_M = <X^\sigma: \sigma \not \in \mathcal{F}>,\] where $X^\sigma=\prod_{e \in \sigma}X_e$. We refer  to~\cite{Herzog11} for the study of such ideals. The Stanley-Reisner ring of the matroid is then the quotient $S_M=S/I_M$. This ring has minimal $\mathbb{N}^{|E|}$ and $\mathbb{N}$ graded free resolutions and as described in~\cite{JV2}, they are of the form  \[0 \leftarrow S_M \leftarrow S \leftarrow \bigoplus_{{\alpha \in N_1}}S(-\alpha)^{\beta_{1,\alpha}} \leftarrow \cdots \leftarrow \bigoplus_{\alpha \in N_{n-\r(M)}}S(-\alpha)^{\beta_{|E|-\r(M),\alpha}} \leftarrow 0\] and %\begin{equation} \label{res}
\[0 \leftarrow S_M \leftarrow S \leftarrow \bigoplus_{j \in \mathbf{N}}S(-j)^{\beta_{1,j}} \leftarrow \cdots \leftarrow \bigoplus_{j \in \mathbf{N}}S(-j)^{\beta_{|E|-\r(M),j}} \leftarrow 0.\] 
%\end{equation}
It is known, in particular, that the numbers $\beta_{i,\alpha}(M)=\beta_{i,\alpha}$  are independent of the minimal free resolution, and when the simplicial complex comes from a matroid (as in our case), also independent of the field $\mathbb{K}$. As one sees $\beta_{i,j}(M)=\beta_{i,j}=\Sigma_{|\alpha|=j}\beta_{i,\alpha},$ for all $i,j$.

Such Betti numbers, for matroids being specified later, will be instrumental for our main results. 
\begin{definition}
 For $i,j,l$ in question, given a matroid $M$, we let $\beta^{(l)}_{i,j}(M)$  be the $\mathbb{N}$-graded graded Betti number of the $l^{\text{th}}$ elongation matrix $M^{(l)}$.
 \end{definition}
 In particular, $\beta^{(0)}_{i,j}(M)=\beta_{i,j}(M) ($ If $M$ is fixed, we just write $\beta^{(l)}_{i,j}$ for $\beta^{(l)}_{i,j}(M)$.) 

\subsection{$q$-Matroids and $(q,m)$-polymatroids}
%We let $E$ be the vector space $\Fq^n$ over $\Fq$ and define:
%$$\Sigma(E) = \text{the set of all }\Fq\text{-linear subspaces of } E.$$
%We usually take $E=\mathbb{F}_q^n$. In this case, there is a standard ``dot product" on $E$ and
For $X\in \Sigma(\E)$, we denote by $X^{\perp}$ the dual of $X$ (with respect to the standard dot product), i.e., 
$X^{\perp} = \{\mathbf{x} \in \E: \mathbf{x}  \cdot \mathbf{y}  = 0 \text{ for all } \mathbf{y} \in X\}$. % be the dual of $X$.  
It is elementary and well-known that $X^{\perp}\in \Sigma(\E)$ with $\dim~ X^{\perp} = n - \dim~ X$ and $(X^{\perp})^{\perp} = X$, although $X\cap X^{\perp}$ need not be equal to $\{\mathbf{0}\}$, but of course $\E^\perp = \{\mathbf{0}\}$. %unless $X=E$.   

%We start with giving the essential
The first part  of the  following key notion %definition %of a central object of this work 
is due to \cite[Definition 2]{S}. 

\begin{definition}\label{FirstDef}

%(i) 
A $(q,m)$-{\em polymatroid} is an ordered pair $\mathcal{M}= (\E,\rho)$ consisting of
the vector space $\E=\mathbb{F}_q^n$ %over $\mathbb{F}_q$
 and a function $\rho:\Sigma(\E)\to\mathbb{N}_0$
%(where $\Sigma(E)$ is the set of all subvector spaces of $E$) 
satisfying (P1)--(P3) below: 
% for all $X,Y \in \Sigma(E)$:  %linear subspaces $X$ and $Y$ of  $E$:
\begin{list}{}{\leftmargin=1em\topsep=1.5mm\itemsep=1mm}
\item[{\rm (P1)}]  $0\leq \rho(X) \leq m\dim X$ for all $X \in \Sigma(\E)$;
\item[{\rm (P2)}] $\rho(X) \le \rho(Y)$ for all $X,Y \in \Sigma(\E)$ with $X \subseteq Y$;
%If $X \subseteq Y$, then 
\item[{\rm (P3)}] $\rho(X+Y)+\rho(X \cap Y) \le \rho(X)+\rho(Y)$, for all $X,Y \in \Sigma(\E)$.
\end{list}
\end{definition}

 The function $\rho$ is called the rank function of the $(q,m)$-polymatroid. The rank of a $(q,m)$-polymatroid $\mathcal{M}=(\E,\rho)$ is $rank(\mathcal{M} = \rho(\E)$. The nullity function $\eta$ of the $(q,m)$-polymatroid is given by $\eta(X) = m \dim_{\Fq} X-\rho(X)$ for $X \in \Sigma(\E)$. By (P1), this is an integer-valued non-negative function on $\Sigma(\E)$.

To us the following special case will be the most important one:
\begin{definition}
A $q$-matroid is a $(q,1)$-polymatroid.
\end{definition}
\begin{definition}\label{def:dual_qmatroid}
Let $\MM=(\E,\rho)$ be a $(q,m)$-polymatroid. The dual $(q,m)$-polymatroid  of $\MM$ is the $(q,m)$-polymatroid $\MM^*=(\E,\rho^*)$ with \[\rho^*(X) = m\dim X +\rho(X^{\perp})-\rho(\E).\]
The dual of a $q$-matroid is obtained from the case $m=1.$
\end{definition}
\begin{definition}
Let $\mathcal{M}=(\E,\rho)$ be a $q$-matroid. Then a subspace $F \subseteq \E$ is called a \textit{$q$-flat} if $\rho(F \oplus \langle e \rangle) > \rho(F)$ for all $e \in E \backslash F$.
\end{definition}

 \begin{definition}
Let $\mathcal{M}=(\E,\rho)$ be a $q$-matroid, and $\eta$ be its nullity function, i.e., $\eta(X) = \dim_{\Fq}X-\rho(X)$ for all $X \in \Sigma(\E)$. For $0 \leqslant i \leqslant \eta(\E)$, a subspace $X \subseteq \E$ of nullity $i$ is called a \emph{$q$-cycle} of $\mathcal{M}$ if $X$ is  minimal in $\mathcal{N}_i$ w.r.t. inclusion, where \[\mathcal{N}_i =\{X \in \Sigma(\E),\ \eta(X)=i\}.\] The $q$-cycles of nullity $1$ are called $q$-circuits.
\end{definition}
\begin{remark}
$(q,m)$-flats and $(q,m)$-cycles can be defined in analogous ways for 
$(q,m)$-polymatroids in general, but we will only treat the case $m=1$
in what follows.
\end{remark}

We then have:
\begin{lemma} \label{fundamental}
Let $\mathcal{M} = (\E,\rho)$ be a $q$-matroid. Then $X \in \Sigma(\E)$ is a $q$-flat (of rank $r$) of a $q$-matroid $\mathcal{N}$ if and only if its orthogonal complement $X^{\perp}$ is a $q$-cycle (of nullity $\rho(\mathcal{M})-r$) for $\mathcal{M}^*$.
\end{lemma}

\begin{proof}  Let $X$ be a $q$-flat of a $q$-matroid $\mathcal{M}$ of rank $r$. From the definition of $q$-cycle, any subspace strictly contained in a $q$-cycle $A$ have nullity strictly less the nullity of $A$, which we assume to be $a$. 

Then from the identity, $ \rho^*(X) = \dim~ X + \rho(X^\perp) - \rho(\E)$, it is clear that if any space of the form $X\oplus \langle v\rangle$ has rank strictly greater than that of $X$, the nullity of any subspace $B$ of $X^\perp$ of codimension $1$ should be strictly one less than that of $X^\perp$. So if $X$ is a $q$-flat, then  $X^\perp$ is a $q$-cycle.
Similar logic proves the converse also.

Since $\eta^*(X^\perp) = \dim~ X^\perp - \dim~ X^\perp  - \rho (X) + \rho(\E) = \rho(\mathcal{M}) - \rho(X)$, it shows that for a $q$-flat $X$ of rank $r$ its orthogonal complement $X^\perp$ is a $q$-cycle of $\mathcal{M}^*$ of nullity $\rho(\mathcal{M}) - r$. 
\end{proof}
We now introduce an important definition, essentially taken from \cite{JP}:
\begin{definition} \label{class}
 For every $q$-matroid $\mathcal{\mathcal{M}}=(\E,\rho)$, we  associate a pair 
$$Cl(\mathcal{M})=(P(\E),\r_{\rho}),$$ 
where $P(\E)$ is the set of all $1$-dimensional subspaces of $\E=\Fq^n$, in other words, $P(\E)$ is the projectivization of $\E$ with $ |P(\E)| = \frac{q^n-1}{q-1}=q^{n-1}+q^{n-2}+\cdots+1.$ Moreover, for $S \subseteq P(\E)$, we set:
$$\r_{\rho}(S):= \rho(\langle S \rangle),$$ 
where $\langle S \rangle$ is the (affine) subspace in $\E=\Fq^n$, spanned by the $1$-dimensional subspaces of $\E$ in $S$.
\end{definition} 
\begin{proposition} \label{matclass}
For every $q$-matriod $\mathcal{M}=(\E,\rho)$, the pair    
$Cl(\mathcal{M})=(P(\E),\r_{\rho})$ is a matroid.
\end{proposition}
\begin{proof}
Properties (R1) and (R2) follow directly from the properties (P1) and (P2) of Definition \ref{FirstDef}, respectively. To prove $(R3)$, we take $X,~Y \subseteq P(\E)$ and verify:
\begin{align*}
\r_{\rho}(X \cap Y)+\r_{\rho}(X \cup Y)&=\rho(\langle X \cap Y \rangle + \rho(\langle X \cup Y  \rangle)\\
                                     &=\rho(\langle X \rangle \cap \langle Y \rangle)+\rho(\langle X\rangle+ \langle Y \rangle)\\
                                     &\le \rho(\langle X \rangle) +\rho(\langle Y \rangle) \qquad (\text{by P3})\\
                                     &\le \r_{\rho}(X)+\r_{\rho}(Y).
\end{align*} 
\end{proof}

Furthermore, we observe:
\begin{lemma} \label{transfer}
 Let $\mathcal{M} = (\E,\rho)$ be a $q$-matroid and $Cl(\mathcal{M}) = (P(\E), \r_{\rho})$ be the classical matroid associated to $\mathcal{M}$. Then $S \subseteq P(\E)$ is a flat with rank $\r(S)=r$ for $Cl(\mathcal{M})$ if and only if it is of the form $P(F)$ for some $q$-flat $F$ of $\mathcal{M}$ of $q$-rank $\rho(F)=r$.
\end{lemma}
\begin{proof}
If $X=P(F)$ for  some $q$-flat $F$, and $y \in P(\E)-X$,
then $\r(X \cup \{y\}) = \rho(F \oplus \langle y \rangle)>\rho(F)=\r(X),$ 
so $X$ is a flat.

If $X$ is not of the form $P(F)$ for $F$ a $q$-flat, then  $X=P(G)$ for $G$ not a $q$-flat, or $X$ is not of the form $P(H)$ for any $G \in \Sigma(\E)$.
If $X=P(G)$ for $G$ not a $q$-flat, then there exists $e$ such that $\rho(G \oplus \langle e \rangle)=\rho(G)$.
Then $\r(P(G \oplus \langle e \rangle))=\rho(G \oplus \langle e \rangle)=\rho(G)=\r(X)$, so $X$ is not a flat.

If $X$ is not of the form $P(G)$ for any $G\in \Sigma(\E)$, then $X$ is strictly contained in $Y=Span(X)$ in $P(\E)$. Hence: $\r(X)=\rho(\langle X \rangle)=\rho(\langle Y \rangle)=\r(Y)$, and so $X$ is not a flat.
\end{proof}

\begin{definition}\label{def:interval}Let $(E,\mathcal{R})$ be a poset. The  opposite of a poset $(E,\mathcal{R})$ is the poset $(E,\mathcal{S})$ where $x \mathcal{S} y \Leftrightarrow y \mathcal{R} x.$ %Some authors use the term "dual" instead of "opposite".

%Let $a,b \in E$ with $a \mathcal{R} b$. The interval $[a,b]$ is the set \[[a,b] = \left\{x \in E:\ a \mathcal{R} x \textrm{ and } x \mathcal{R} b\right\}.\]
\end{definition}

%Together with the order relation $\mathcal{R}$, the interval $[a,b]$ forms a poset. It has a unique maximal/greatest element, namely $b$, and a unique minimal/least element, namely $a$.

\begin{definition}\label{def:chain} 
\begin{enumerate}
\item [\rm{(i)}] Let $(E,\mathcal{R})$ be a finite poset. A chain $C$ in $E$ is a totally ordered subset of $E$ (meaning $a \mathcal{R} b$ or $b \mathcal{R} a$ for $a,b$ in $C$). 

\item [\rm{(ii)}] The length of a chain is equal to the cardinality of the chain minus $1$. The length of a finite poset is the maximal length of chains in the poset. 

\item [\rm{(iii)}] If the poset has the Jordan-Dedekind property
(meaning: all maximal chains have the same length), then the rank of an element $x \in E$ is the length of the poset $([0,x],\mathcal{R})$.
\end{enumerate}
\end{definition}

 \begin{definition}\label{def:lattice} A finite lattice is a finite poset $P=(E,\mathcal{R})$, where there exists a maximal element, denoted by $1$, a minimal element, denoted by $0$, and for any two elements $a,b \in E$, there exists a least upper bound (or join) $a \vee b$ and a greatest lower bound (or meet) $a \wedge b$. An atom is a minimal element of the subset $E\backslash\{0\}$.
\end{definition}

%\begin{remark} \label{JDf}
%From now on we will assume that all posets and lattices appearing will be finite, and that they have the Jordan-Dedekind property. 
%\end{remark}

The opposite lattice $P^*$ of a lattice $P$ satisfies $0_{P^*}=1_P$, $1_{P^*}=0_P$, $a \vee_{P^*} b = a \wedge_P b$ and $a \wedge_{P^*} b = a \vee_P b$.

Let $M$ be a matroid on the ground set $E$. It is well known that the set of flats of $M$ is a lattice, where the order is the inclusion order.
Moreover, it is well known that this lattice has the Jordan-Dedekind property, and therefore has a well-defined rank function. The minimal element of the lattice is the closure of $\emptyset$, its maximal element is  $E$, while the meet of two flats is their intersection, and the join is the closure of their union. 
\begin{definition}
A geometric lattice is a finite lattice having the Jordan-Dedekind property, and where its rank function, say $r$, satisfies:
\begin{itemize}
    \item It is atomistic, i.e., every element is a supremum of a set of atoms.
    \item It is semimodular, i.e.,$r(x \wedge y) + r( x \vee y) \le r(x) + r(y).$
\end{itemize}
\end{definition}
It is well known that the flats of a matroid constitute a geometric lattice, where the rank function of the lattice corresponds to the rank function of the matroid, for each flat.

Lemma \ref{transfer} has the following immediate, probably well known, consequence:
\begin{proposition} \label{compare2}
\begin{itemize}
\item [{\rm (a)}] As posets ordered by inclusion, the set of flats of $Cl(\mathcal{M})$, and the set of $q$-flats of $\mathcal{M}$ are isomorphic.
\item [{\rm (b)}] The posets of cycles of $Cl(\mathcal{M})^*$ and $q$-cycles of 
$\mathcal{M}^*$ are isomorphic.
\item [{\rm (c)}] The poset of $q$-flats of $\mathcal{M}$ constitute a geometric lattice.
\end{itemize}
\end{proposition}

\begin{proof}
\begin{itemize}
\item [{\rm (a)}] From Lemma \ref{transfer} there is a bijection between the flats of $Cl(\mathcal{M})$ and the flats of $\mathcal{M}$. Since it is  inclusion-preserving, the lattices are isomorphic.
\item [{\rm (b)}] This holds since by (a) their opposite lattices are isomorphic.
\item [{\rm (c)}] This holds by (a), since the flats of $Cl(\mathcal{M})$ are well known to do so.
\end{itemize}
\end{proof}
\section{Generalized  Weights of $q$-matroids and Betti numbers}
\label{sec3}

 In this section, we consider rank-metric codes and their corresponding $(q,m)$-polymatroids and $q$-matroids. The main result of this section gives a relation between generalized rank weights of Gabidulin rank-metric codes with Betti numbers of certain classical matroids associated to the $q$-matroids corresponding to the codes.

Let $\Mat$, or simply  $\M$, denote the space of all $m\times n $ matrices with entries in the finite field $\Fq$. Note that $\M$ is a vector space over $\Fq$ of dimension $mn$. 

{\begin{definition} \label{found}
\begin{itemize}
\item [\rm{(a)}] By a \emph{Delsarte rank-metric code},  or simply, a \emph{Delsarte code}, we mean a 
$\Fq$-linear subspace of $\M$.  %We write $\dim_{\Fq}\C$, or simply $\dim\C$, to denote the dimension of a Delsarte code $\C$. 
\item [\rm{(b)}] By a \emph{Gabidulin rank-metric code}, or a \emph{vector rank-metric code}, we mean a $k$-dimensional $\F_{q^m}$-linear space $\C$ in $\F_{q^m}^n$.

\end{itemize}
\end{definition}}
A Gabidulin rank-metric code $\C$ is also then  a $K=km$-dimensional linear code over $\Fq$.
This can be viewed as a special case of a Delsarte code in the following way: Fix once and for all an $\Fq$-basis $B=\{b_1,\cdots,b_m\}$ of $\F_{q^m}$. Therefore, any element $a \in \F_{q^m}$ can be uniquely written as $a_1b_1+\cdots a_mb_m$ and represented by a column vector $(a_1,\cdots,a_m)^t$. In a similar way, any codeword $c = (w_1,\cdots,w_n)\in \F_{q^m}^n$ can be represented as an $m \times n$-matrix
$M_B(c)=\begin{pmatrix} w_{1,1} & \cdots & \cdots& w_{n,1} \\  
                \cdots&  & \cdots &\cdots \cdots\\ 
                \cdots&  & \cdots &\cdots \cdots\\
                w_{1,m} & \cdots & \cdots& w_{m,n}
\end{pmatrix}$.

  Following Shiromoto \cite{S}, we associate to a Delsarte code (resp. Gabidulin rank-metric code) $\C$, a family $\{\C(X):X \in \Sigma(\E)\}$ of subcodes of $\C$, and a $(q,m)$-polymatroid (resp. $q$-matroid) as follows.

%The following definitions are taken from \cite{S}. We start with
%\begin{definition}
%A Delsarte rankmetric code  (or just a Delsarte code $\C$) is a $K$-dimensional linear subspace of the $\mathbb{F}_q$-linear vector space of $(m \times n)$-matrices over $\mathbb{F}_q$, for some $K=0,1,\cdots,mn$.
%\end{definition}
%This leads us to the following 

\begin{definition}\label{CXrhoC}
\begin{enumerate}
\item[{\rm (a)}] Let $\C \subseteq \Mat$ be a Delsarte code. Given any $X\in \Sigma(\E)$, $\C(X)$ is defined to be the subspace of $\C$ consisting of all matrices in $\C$ with row spaces contained in $X$. Let $\rho^{1}_\C$ : $\Sigma(\E) \longrightarrow \mathbb{N}_0$ be the function defined as 
$$ \rho^{1 }_\C(X)=\dim_{\mathbb{F}_q}\C-\dim_{\mathbb{F}_q}\C(X^{\perp}) \quad \text{ for } X \in \Sigma(\E).
$$ 
 
%For an $\mathbb{F}_q$-linear subspace $X$ of $\mathbb{F}_q^n$, we denote by $\C(X)$ the linear subcode of $\C$,
 %\item[{\rm (b)}]{By $\rho^{1 }_\C$ we denote the function from $\Sigma(\E)$ to ${\mathbb{N}_0}$ defined by 
%The polymatroid $P(C)=(\E,\rho)$ associated to a Delsarte code $\C$ is the one with $E=\mathbb{F}_q^n$, and

\item [{\rm (b)}] Let $\C$ be a Gabidulin rank-metric code over $\Fqm/\Fq$. With respect to the fixed $\Fq$-basis $B$ of $\Fqm$, we consider the respresentation of codewords by matrices. Then for any $X \in \Sigma(\E)$, $\C(X)$ is defined similar to the case of Delsarte codes. Let $\rho_\C$ : $\Sigma(\E) \longrightarrow \mathbb{N}_0$ be the function defined as 
$$\rho_{\C}(X)=\dim_{\mathbb{F}_{q^m}}\C-\dim_{\mathbb{F}_{q^m}}\C(X^{\perp}).$$  
%[{\rm (c)}] $d(A,B)=\rk (A-B),$ for two elements $A,B \in \C$.}
\end{enumerate}
\end{definition}  

\begin{proposition}\label{q-mat}
The pair $(\E,\rho^{1}_\C)$ is a $(q,m)$-polymatroid,
and the pair $(\E,\rho_{\C})$ is a $q$-matroid.
\end{proposition}
\begin{proof}
The first statement is proved in \cite{S}, and the second one then follows automatically, since  $\rho_{\C}(X)=\frac{\rho^1_\C(X)}{m}$.
\end{proof}
It is important to note that $q$-matroids associated to Gabidulin rank metric codes were first defined in \cite[Definition 22 and Theorem 24]{JP2017} and the $q$-matroid in \cite{JP2017} is same as $(\E,\rho_{\C})$ in Proposition \ref{q-mat}.           
\begin{remark}\label{nuCstar}
{\rm
%It is shown in \cite[Proposition 3]{S} that $(\E,\rho^{1 }_\C)$ does indeed satisfy all the axioms of a $(q,m)$-polymatroid. It then follows immediately  that the pair $(\E,\rho)$ is a $q$-matroid.
 We note that the nullity function $(\eta_{\C}^1)^* $ of the dual $(q,m)$-polymatroid of $(\E,\rho_C^1)$ satisfies:
\begin{equation}\label{nu1C}
(\eta_{\C}^1)^* (X) = \dim_{F_{q}}\C(X) \quad \text{for } X \in \Sigma(\E).
\end{equation}
for a Delsarte code, and
that the nullity function $(\eta_\C)^*$ of the dual of the $q$-matroid $(\E,\rho_{\C})$ then is given by 
\begin{equation}\label{nuC}
(\eta_\C)^* (X) = \dim_{F_{q^m}}\C(X) \quad \text{for } X \in \Sigma(\E).
\end{equation}
for a Gabidulin rank-metric code.}
\end{remark}

%\begin{example} \label{Gabidulin}
%{\rm Let $\C$ be a Gabidulin code, that is: $\C$ is a $k$-dimensional $\F_{q^m}$-linear code in $\F_{q^m}^n$.
%Then $\C$ is also a $K=km$-dimensional linear code over $\Fq$. Fix once and for all a basis $B=\{b_1,\cdots,b_m\}$ of $\F_{q^m}$ as a vector space over $\Fq$. Any element $a \in \F_{q^m}$ can then be written as $a_1b_1+\cdots a_mb_m$ and represented by a column vector $(a_1,\cdots,a_m)^t$. Furthermore any codeword $(w_1,\cdots,w_n)\in \F_{q^m}^n$ can in a similar way be represented by an $m \times n$-matrix:
%$\begin{pmatrix} w_{1,1} & \cdots & \cdots& w_{n,1} \\  
%                \cdots&  & \cdots &\cdots \cdots\\ 
%                \cdots&  & \cdots &\cdots \cdots\\
%                w_{1,m} & \cdots & \cdots& w_{n,m}
%\end{pmatrix}$

%Furthermore we have in principle two different rank functions associated to $\C$:
%\begin{itemize}

%\item First: $\rho_\C^1(X)$, for any $X \in \Sigma$

%\item Moreover it is clear both $\C$ (by definition) and %$\C(X^{\perp})$ are $\FQ$-linear spaces for $Q=q^m$, so 
%$\rho_\C^1(X)=\dim_{\Fq}\C-\dim_{\Fq}\C(X^{\perp})=m\rho_\C(X)$ is divisible by $m$. Thus we define our second, "normalized", rank function, also this on $\Sigma,$ to be simply 

%$\rho_\C(X)=\frac{\rho_\C^1(X)}{m}= \dim_{F_{q^m}}\C-\dim_{F_{q^m}}\C(X^{\perp})$. 

%It is clear that $(\E,\rho_\C)$ is a $q-$matroid. %according to standard definitions by [JP].
%See for example \cite{GJLR}.
%\end{itemize}

%\begin{definition}
%For any Gabidulin code $\C$ we set $\MC=(\E,\rho_\C)=
%(\Fq^n,\frac{\rho_\C^1}{m})$. 
%\end{definition} 
\begin{definition}
Let $\C \subseteq \Fqm^n$ be a Gabidulin rank-metric code of length $n$ and dimension $k$ over the extension $\Fqm/\Fq$. Let $d$ be the minimum rank distance of $\C$. If $d = n-k+1$, then $\C$ is called a maximum rank distance (MRD) code. 
\end{definition}

Next we recall the very first construction of MRD codes, independently given by Gabidulin \cite{Gab85} and Delsarte \cite{Del}. 
\begin{example}\label{MRDexam}
Let $m,n,k$ be positive integers such that $m \ge n \ge k$. If we consider $\mathbf{a} = (a_1, \ldots, a_n) \in \Fqmn$ such that $\{a_i : i = 1, \ldots, n \}$ is a linearly independent set over $\Fq$, then $\Fqm$-linear code generated by the following matrix $\mathbf{G}$ is an MRD code of length $n$ and dimension $k$. 
	\[
	\mathbf{G} := 
	\begin{pmatrix}
	a_1 & a_2 & \hdots & a_n \\
	a_1^q & a_2^q & \hdots & a_n^q \\
	\vdots & \vdots & \ddots & \vdots \\
	\quad a_1^{q^{k-1}} & \quad a_2^{q^{k-1}} & \hdots &\quad a_n^{q^{k-1}} \\
	\end{pmatrix}.
	\]

(See, for example, \cite{Gab85} and \cite[\S\, 2]{ravagnani15} for the definition and basic facts about MRD codes.)
\end{example}

\begin{example} \label{MRD}
{\rm 
{Assume for simplicity that $m \ge n$. Let $\C \subseteq \Mat$ be an MRD code of dimension $K$ over $\Fq$. Such a $\C$ is a Delsarte code 
%with the properties 
%An MRD-code is a rank-metric code as defined above, with the extra properties 
such that $K = \dim_{F_q}\C$ is divisible by $m$ and 
$\C(X)=\{0\}$ for all subspaces $X$ of $E$ with $\dim_{F_q}X \le n-\frac{K}{m}$. The latter follows, for instance, from \cite[Proposition 6.2]{GJLR}. Thus $\rho^{1}_\C(Y) =K$ if $Y\in \Sigma(\E)$ with $\dim_{F_q}Y \ge K/m$, and hence $\rho_\C(Y) =k.$ 
Further, in view of \cite[Theorem 6.4]{GJLR}, we see that  $\rho_\C(Y) =\dim_{F_q}Y$ if $Y\in \Sigma(\E)$ with $\dim_{F_q}Y \le K/m$.. 
It follows that $(\E,\rho_{\C})$ is the uniform $q$-matroid $U(k,n)$
(i.e. the $q$-matroid with rank function $\rho(X)=\min\{\dim(X),k\}.$} }

\end{example}

%\begin{example} \label{MRD}
%{\rm For simplicity, let us assume that $m > n$. An MRD-code is a rank-metric code as defined above, with the extra properties that $K=\dim C$ is divisible by $m$, and 
%$\C(X)=\{0\}$, for all subspaces $X$ of $E$ of  dimension (at most) $n-\frac{K}{m}$. One sees that $P(C)=U(\frac{K}{m},n)$. 
%For a more detailed description of MRD codes, covering all possibilities for $m,n$, see \cite{GJLR}.}
%\end{example}
 Now we recall the definition of generalized rank weights of Gabidulin rank-metric codes in terms of nullity function of the corresponding (dual) $q$-matroids. The following definition is from \cite{JP2017} where the authors prove its equivalence with the other existing definitions. 
\begin{definition}\cite[Definition 5]{JP2017}\label{gen.weight}
Let $\C$ be a Gabidulin rank-metric code over $\Fqm/\Fq$. 
\begin{enumerate}
     \item [\rm{(a)}]  The rank support $Rsupp(c)$ of any codeword $c \in \C$ is the $\Fq$-linear row space of $M_B(c)$ (the matrix representation of $c$, as in the text following Definition \ref{found}.) For a subcode $\D$ of $\C$, its rank support $Rsupp(\D)$ is the $F_q$-linear span of the set $\{Rsupp(d) : d \in \D\}$.
    \item [\rm{(b)}] Let $\dim_{\Fqm}{\C} = k$. Then for $1 \leq r \leq k$, the $r$-th generalized rank weight of $\C$ is defined as \[d_r(\C)=\min\{\dim_{\F_{q}}Rsupp(\D), \D \text{ is a subcode of } \C \text{ with } \dim_{\Fqm}(\D)=r\}.\] 
    %The rank distance between two matrices $A,B$ in $\mathbb{M}$ is $\dim_{\Fq}~Rsupp(A-B).$ With this distance functions Delsarte (and Gabidulin) codes are translation-invariant metric spaces.
  %For a subcode $D$ of $\C$, its rank support $Rsupp(D)$ is the $F_q$-linear span of the set $\{Rsupp(A) : A \in D\}$.
\end{enumerate}
\end{definition}
\begin{remark}
In this article, we define the rank support using row spaces, irrespective of $m \geq n$ or $m<n.$ This is following the definition by Shiromoto in \cite{S} (where one uses column spaces, but of the transposes of our matrices)  and thereafter, in \cite{GJ20}. As proved in \cite[Theorem 37]{GJ20}, this $d_r(\C)$ matches with Ravagnani's definition of generalized weights for Delsarte codes in \cite{ravagnani16} only if $m > n$.
\end{remark}
%Following Definition \ref{CXrhoC}, it is clear that 
 Then we express the generalized rank weights in terms of the nullity function of dual $q$-matroid associated to the code :

\begin{proposition} \label{weights}
For a Gabidulin rank-metric code $\C$ over $\Fqm/\Fq$, the $r$-th generalized rank weight is
$${d}_r= \min \{\dim_{\Fq} X : X \in \Sigma \text{ with } \eta_{\C}^*(X) \ge r \}.$$
\end{proposition}
\begin{proof}
Following the definitions of $\C(X)$ in Definition \ref{CXrhoC} and $d_r$ in Definition \ref{gen.weight}(b), we have 
$$d_r(\C) = \min \{\dim_{\Fq}X : X\in \Sigma(\E) \text{ with } \dim_{\Fqm }\C(X) \ge r\}.$$ Therefore, the statement follows from Remark \ref{nuCstar}, since $\eta_{\C}^*(X) =\dim_{\Fqm }\C(X)$.
\end{proof}
{Inspired by Remark \ref{nuCstar} and Proposition \ref{weights} we have:
\begin{definition} \label{dr}
 For any $q$-matroid $\MM=(\E,\rho)$ we set 
\[d_r(\MM)= \min \{\dim_{\Fq} X : X \in \Sigma \text{ with }   \eta^*(X) \ge r ,\}\]
where $\eta^*$ is the nullity function of the dual $q$-matroid ${\MM}^*$.
\end{definition}}

It is clear from the above description that, for a Gabidulin rank-metric code $\C$ (or more generally, for a $q$-matroid), the $r$-th generalized weight is equal to
the smallest $\Fq$-dimension of any $q$-cycle of nullity $r$ of ${\MM_{\C}}^*$. As an immediate consequence, we obtain:
%$\overline{d}_r=\min \{\dim_{\Fq} X : X\in P=\Sigma(\E) \text{ with }   \dim _{F_{q^m}} \C(X) \ge r , \}$ and 
 %\[ \min \{\dim_{\Fq} X : X \in P=\Sigma(\E) \text{ with }   (\eta)^*(X) \ge r ,\},\] where $(\eta)^*$ is the conullity of the "normalized" rank function $\rho_C$,
 
\begin{lemma} \label{flatweight}
 Let $\C$ be a Gabidulin rank-metric code over $\Fqm/\Fq$ of dimension $k$. Then for any $1 \leq r \leq k$, ${d}_r=n-m_r$, where $m_r$ is the largest dimension over $\Fq$  
of any $q$-flat of rank $k-r$ for $\MC$.
\end{lemma}

\begin{proof} 
This follows from Definition \ref{dr} and Lemma \ref{fundamental}.
\end{proof}

 	 	%We will use this definition in Section \ref{ncwq}.
 	 	%In the geometric lattice of flats passing from 
 	 	%$M$ to $M_{(l)}$ means substituting all flats of rank at least $k-l$ with a single "one" node. The passage from $M$ to $M^{(l)}$, on the other hand, means substituting all flats of rank at most $l$ with a single "zero" node.

%We obtain, using the notation above:
The next result provides a relation between the generalized rank weights of a Gabidulin rank-metric code $\C$ and the cycles of the associated dual matroid $Cl(\MC)^*$ of the corresponding $q$-matroid $\MC$.
\begin{corollary} \label{closingin}
{Let $Cl(\mathcal{M})$ be the classical matroid associated to a $q$-matroid $\mathcal{M} = (\E,\rho)$.}
\begin{itemize}
\item [{\rm (a)}] Any cycle $X$ of $Cl(\mathcal{M})^*$ is the complement of projective spaces
(when interpreting $P(\E)$ as projective $(n-1)$-space). Its cardinality is $q^{n-1}+q^{n-2}+\cdots+q^m$, for $m$ the dimension of the flat $F$ for for which the $X$ is the complement of $Cl(F)$.
\item [{\rm (b)}] For a Gabidulin rank-metric code $\C$ we have that 
   ${d}_r$ is equal to the smallest $j$ such that 
   there exists a cycle of nullity $r$ and cardinality 
   $[j]$ for the matroid $(Cl((\E,\rho_\C))^*$.
\end{itemize}
\end{corollary}

\begin{proof}
Part (a): This follows from Lemma \ref{transfer}. 
Proof of (b)%(Set $M=\MC$: 
\begin{align*}
{d}_r &= \min\{\dim_{\Fq}X | \eta^*(X) = r\}\\
&= \min\{\dim_{\Fq} X |X \text{ is a $q$-cycle of } ({\MC})^* \text{ of nullity } r\}\\
&= \min\{j |X^{\perp} \text{ is a $q$-flat of } \MC \text{ of rank } k-r \text{ and dimension } n-j\},\\
&\text{ and, using Lemma \ref{transfer} again:}\\
&=\min\{j |P(X^{\perp}) \text{ is a flat of } Cl(\MC) \text{ of rank } k-r \text{ and cardinality } q^{n-j-1} +\cdots+1\}\\
&=\min\{j |P(X^{\perp})^c \text{ is a cycle of } Cl(\MC)^* \text{ of nullity } r \text{ and cardinality } [j]\}.
\end{align*}
%From part (a) we know that any cycle $\sigma$ of $Cl(\mathcal{M})^*$ is of cardinality $q^{n-1}+\ldots + q^m$, where $\sigma^c$ is a flat of $Cl(\mathcal{M})$ which spans a flat of $M$ of dimension $m$. So we can write, 
%\begin{align*}
%\overline{d}_r &= \min\{j : \exists \text{ a cycle of } Cl(\mathcal{M})^* \text{ of nullity } r \text{ and cardinality } q^{n-1}+\ldots+q^{n-j}\}\\
%&=\min\{j | \sigma \in N_r(Cl(\mathcal{M})^*) \text{ and } %|\sigma|=q^{n-1}+\ldots+q^{n-j} \}.
%\end{align*}
\end{proof}

 We use this relation and the following result from \cite{JV2} about classical matroids to express the generalized rank weights in terms of certain Betti numbers.

\begin{theorem}\cite[Theorem 2]{JV21} \label{thm:JV21}
{ Let $M = (E, \r)$ be a matroid on a finite set $E$ and $N_i(M)$ the set of cycles of $M$ of nullity $i$. Then
%Suppose for a $\sigma \subseteq \E$, $\beta_{i, \sigma}$ denotes the $\mathbb{N}$-graded Betti number of the Stanley-Reisner ring associated to $M$. Then
\[\beta_{i, \sigma}(M) \neq 0 \text{ if and only if } \sigma \in N_i(M).\]}
%where $N_i$ is the set cycles of $M$ of nullity $i$. }
\end{theorem}

%Finally we study the Stanley-Reisner ring associated to the independence complex of the classical matroid $(Cl(\E,\rho_\C))^*$. From Corollary \ref{closingin} and the main result of \cite{JV2} we obtain:

\begin{theorem} \label{betti}
For a Gabidulin rank-metric code $\C$ we have that:
\begin{itemize}
\item [{\rm (a)}] \[{d}_r= \min \{j | \beta_{r,[j]} \ne 0\},\]
for the $\mathbb{N}-$graded Betti numbers of  the Stanley-Reisner ring associated to the independence complex of the classical matroid $Cl(\MC)^*$
\item [{\rm (b)}] These $\mathbb{N}-$graded Betti numbers satisfy
$\beta_{r,s}=0$, for all $s \le \frac{q^n-1}{q-1}$ where $s$ is not of the form $[j]$ for some $j$.
\end{itemize}
\end{theorem}
 
 \begin{proof}
 { From Corollary \ref{closingin} we have,
 can write, 
\begin{equation}\label{eq:2}
{d}_r =\min\{j | \sigma \in N_r(Cl(\mathcal{M})^*) \text{ and } |\sigma|=q^{n-1}+\ldots+q^{n-j} \}.
\end{equation}

Now Theorem \ref{thm:JV21} implies that $\sigma$ is a cycle of $Cl(\MC)^*$ of nullity $r$ and of cardinality $[j]$, if and only if the $\mathbb{N}$-graded Betti number $\beta_{r,[j]}$ of the associated Stanley-Reisner ring is nonzero. Thus (a) follows directly from the expression of ${d}_r$ in equation \eqref{eq:2}. 
  
 To prove (b), first we recall from Corollary \ref{closingin} that all the cycles of $Cl(\MC)^*$ are of cardinality $[j]$ for some $j$ with $1 \le j \le n$. Now from Theorem \ref{thm:JV21}, it is clear that the $\mathbb{N}-$graded Betti numbers $\beta_{r,s}$ are zero if $s$ is not of the form $[j]$ for some positive $j$ with $j \le n$. 
 }
 \end{proof}
  	 	\section{Number of codewords of each rank weight via classical matroids} \label{sec4}
        For this section we introduce some notations and fix some parameters.
 	 	Let $m,n$ be positive integers and $\C \subseteq \F_{q^m}^n$ be a Gabidulin rank-metric code over $\F_{q^m}/\Fq$ of dimension $k \le \min\{m,n\}$ with a generator matrix $G = [(g_{i,j})]$ (i.e. a $k \times n$-matrix, with entries in $\F_{q^m}$, and whose row space over $\F_{q^m}$ is $\C$) . Let $ Q = q^m$ and $\tilde{Q} = Q^r$ for some $r \in \mathbb{N}$. Let $\tilde{\C}=\C \otimes_{\F_{Q}} \F_{Q^r}$ denote the extension of $\C$ and thus $\tilde{\C}$ can be considered as a Gabidulin rank-metric code over $\F_{\tilde{Q}}/\Fq$. 
 	 	
 	 	\begin{definition} \label{extension}
 	 	For $0 \leq s \leq k,$ let $A_{\C,s}^{\tilde{Q}}$ denote the number of words of rank weight $s$ in $\tilde{\C}=\C \otimes_{\F_{Q}} \F_{Q^r}$. %We set $\tilde{Q}=Q^r$.
 	 	\end{definition} 
 	 	In this and the following section we will find expressions for $A_{\C,s}^{\tilde{Q}}$.
 	 	First we fix bases for the field extensions considered in this section. Let $\{g_1,g_2,\cdots,g_r\}$ be an arbitrary but fixed basis for $ \FQt$ over $\FQ$. Therefore combining the fixed $\Fq$-basis $\{b_1, \ldots, b_m\}$ of $\FQ$, we then once and for all use $\{b_ig_j\}_{1 \le i \le m,1\le j \le r}$ as a basis for $\F_{\tilde{Q}}$ over $\Fq$. We use the following ordering $$B_r=\{g_1b_1,g_1b_2,\cdots,g_1b_m,g_2b_1,g_2b_2, \cdots,g_2b_m,\cdots,g_rb_1,g_rb_2,\cdots,g_rb_m\}$$ of the $\Fq$-basis of $\F_{\tilde{Q}}$ and consider the representation of codewords in $\tilde{\C}$ as $(mr \times n)$-matrices with entries in $\Fq$ with respect to the ordered basis $B_r$.
 	 	
 	 	%When representing codewords in $\tilde{\C}$ as $(mr \times n)$-matrices with entries in $\Fq$, then row nr. $s$ in such a matrix will refer to element number $s$ in this basis. We may then think of $\{B_r=g_1b_1,g_1b_2,\cdots,g_1b_m,g_2b_1,g_2b_2,\cdots,g_2b_m,\cdots,g_rb_1,g_rb_2,\cdots,g_rb_m\}$ as the ordering of the basis.}
 In this section we express $A_{\C,s}^{\tilde{Q}}$ in terms of the nullity function of a classical matroid associated to the $q$-matroid corresponding to the Gabidulin rank-metric code $\C$.

%
%Another variant of this statement is
\begin{lemma}\label{independent}
 Let $\C \subseteq \FQ^n$ be a Gabidulin rank-metric code over $\FQ/\Fq$ and $\tilde{\C} = \C \otimes_{\F_{Q}}\F_{\tilde{Q}},$ for $\tilde{Q} = Q^r.$ Then we have 
 \[\dim_{\F_Q}\C(U) = \dim_{\F_{\tilde{Q}}}\tilde{\C}(U)= \frac{1}{r}\dim_{\F_Q}\tilde{\C}(U).\]
 \end{lemma}
 \begin{proof}

  The second equality is clear as $\dim_{\F_Q}\F_{\tilde{Q}}= r$.
 
 For the first equality, it is enough to show that $\tilde{\C}(U)$ is isomorphic to $\C(U)^{\oplus r}$ as $\F_Q$-vector spaces. We observe that $\tilde{\C}$, which is given as a row space over $\F_{\tilde{Q}}$ of a matrix with entries in $\FQ$, can be written as a direct sum $Cg_1 \oplus Cg_2 \oplus \cdots  \oplus  Cg_r$. The way we have chosen our basis $B_r$ to express codewords, it is clear that 
$\tilde{\C}(U)$ is a direct sum of  
$\C(U)g_1\oplus \C(U)g_2\oplus \cdots \oplus \C(U)g_r.$ Hence it 
is clear that $\dim_{\FQ}\tilde{\C}(U)=r \dim_{\FQ}\C(U)$.
 \end{proof}
 %To see this, we consider the diagram below where $$\phi: \F_{{Q}}^{r \times n} \longrightarrow (\F_Q^{1 \times n})^r$$ is the $\F_Q$-linear isomorphism which sends an $r \times n$ matrix $M$ over $\F_Q$ to the $r$ tuple with $i$-th coordinate being the $i$-th row of $M$. 
  
 %\begin{tikzcd}[row sep=scriptsize,column sep=scriptsize]
%& \F_{{Q}}^{r \times n}(U) \arrow[dl]\arrow[rr, "\phi|_{\F_{{Q}}^{r \times n}(U)}"] & & (\F_Q^{1 \times n})^r(U)\arrow[dl] \\\F_{{Q}}^{r \times n} \cong  \F_{\tilde{Q}}^n\arrow[rr, "\phi"] & & (\F_{\tilde{Q}}^{n})^r \cong (\F_{{Q}}^{1 \times n})^r\\& \tilde{\C}(U) \cong \tilde{\C} \cap  \F_{{Q}}^{r \times n}(U) \arrow[dl]\arrow[rr] \arrow[uu]&  & {\C}(U)^r \cong \C \cap (\F_{{Q}}^{1 \times n})^r(U) \arrow[dl]\arrow[uu] \\\tilde{\C} \arrow[rr, "\phi|_{\tilde{\C}}"]\arrow[uu] & & \C^r \arrow[to=uu,crossing over]\\
%\end{tikzcd}
%Note that $\C^r$ is the isomorphic image of $\tilde{\C}$ under the restriction of $\phi$.
%All the vertical arrows are inclusion maps and the horizontal ones are isomorphisms induced by $\phi$.
 %Here $\F_{{Q}}^{r \times n}(U) = \{A \in \F_{{Q}}^{r \times n}|~ Rsupp(A)\subseteq U\}$ where $Rsupp(A) = \cup_i \text{ rowsp}(M(A_i))$ and $M(A_i)$ is the matrix representation of the row vector $A_i$ with a fixed basis of $\F_Q$ over $\F_q$. $(\F_Q^{1 \times n})^r(U)$ is defined similarly. From the diagram it is clear that $\tilde{\C}(U) \cong_{\F_Q} \C(U)^r$.
%\end{proof} 

% \begin{proof}
 %(Second version)

  \begin{remark}
 The above lemma implies that $\dim_{\F_{\tilde{Q}}}\tilde{\C}(U)$ is independent of the choice of $r$ so that $\tilde{Q} = Q^r$.
 \end{remark}

\begin{corollary} \label{invariance}
{ Let $\C \subseteq \FQ^n$ be a Gabidulin rank-metric code and $\tilde{\C}$ be the extended code $\C \otimes_{\FQ} \tilde{Q}$ where $\tilde{Q} = Q^r$ for some $r \in \mathbb{N}$. Then the $q$-matroids  $(\E,\rho_{\tilde{\C}})$ corresponding to the codes $\tilde{\C}$ are the same  for any $r \in \mathbb{N}$. }
\end{corollary}
\begin{proof}
For any $U \subseteq \Fq^n$, we have $$\rho_{\tilde{\C}}(U)=\dim_{\F_{\tilde{Q}}} \tilde{\C}-
\dim_{\F_{\tilde{Q}}} \tilde{\C}(U^{\perp})=$$
$$\dim_{\F_{\tilde{Q}}} \tilde{\C}(\E)-\dim_{\F_{\tilde{Q}}} \tilde{\C}(U^{\perp}),$$ which is independent of $r$ by Lemma \ref{independent}.
\end{proof}
\begin{remark} \label{analogue}
 Corollary \ref{invariance} is an obvious consequence if we consider the equivalent definition of rank function of the $q$-matroid associated to Gabidulin rank metric code as given in \cite[Definition 22]{JP2017}. Indeed, $\rho_{\tilde{\C}}(U)= rank (G Y^T)$ where $G$ is a generator matrix of $\C$ and $Y^T$ is the transpose of a generator matrix of $U$. Since all codes 
$\tilde{\C}$ defined as in Definition \ref{extension} have a common generator matrix, the $q$-matroids associated the codes are also same. 
%In our situation, however, we preferred to give the proof above.
\end{remark}
 	 	Now we move onto giving our main result of this section, i.e., an expression for $A_{\C,s}^{\tilde{Q}}$, the number of codewords with rank weight $s$ for $s \in \{1, \ldots, n\}$, using the classical matroids corresponding to the $q$-matroid $(\E,\rho_C)$.  

 First we define the some notions associated to $\tilde{\C}$. 
 \begin{definition}\label{def:tilde}
 Let $\C$ be a Gabidulin rank-metric code over $\F_{Q}/\Fq$ of length $n$ and $\tilde{\C} = \C\otimes_{\FQ}\tilde{Q}$. Then
 for any subspace $U \subseteq \Fq^n$,
 	 	\[
\tilde{\C}(U) := \{\underline{x} \in \tilde{\C}~|~ Rsupp(\underline{x}) \subseteq U\} \text{ and } A_{\C,U}^{\tilde{Q}} := |\{\underline{x} \in \tilde{\C} | Rsupp(\underline{x}) = U\}|. 	 	
 	 	\] 
 \end{definition}

 	 	%Let $A_{\C,U}^{\tilde{Q}} = |\{\underline{x} \in \C\otimes_{\FQ}\tilde{Q}~|~ Rsupp(\underline{x}) = U\}|$. 
 	 	
	 	Then following Definition \ref{extension} and Definition \ref{def:tilde}, we have 
	 	\begin{equation}\label{eq:4}
	 	    A_{\C,n}^{\tilde{Q}} = A_{\C,E}^{\tilde{Q}} \text{ and }
	 	A_{\C,s}^{\tilde{Q}} = \sum\limits_{\substack{U \subseteq \E \\ \dim U = s}} |A_{\C,U}^{\tilde{Q}}|.
	 	\end{equation}
	 	\begin{proposition} \label{help}
	 	Let $\C$ be a Gabidulin rank-metric code over $\F_{Q}/\Fq$ of length $n$. Let $\MC = (\Fq^n, \rho_{\C})$ be its corresponding $q$-matroid and $cl(\MC)=(P(\Fq^n), \r)$ be the associated classical matroid. Then for $\tilde{Q} = Q^r$, 
	 	$$A_{\C,n}^{\tilde{Q}} = (-1)^{{\frac{q^n-1}{q-1}}}\sum\limits_{\gamma \subseteq P(\E)} (-1)^{|\gamma|} \tilde{Q}^{{\nn_{Cl(\mathcal{M}_{\C})}^*}(\gamma)}, $$ where $\nn_{Cl(\mathcal{M}_{\C})}^*$ denotes the nullity function of the dual matroid $Cl(\MC)^*$.
	    \end{proposition}
\begin{proof}
We set $\MC$. Let $U_1, \ldots, U_{\frac{q^n-1}{q-1}}$ be the codimension 1 subspaces of $\E=\Fq^n$. Therefore,
 \begin{align}
 A_{\C,E}^{\tilde{Q}} &=|\tilde{\C|}-|\{\underline{x} \in \C\otimes_{\FQ}\tilde{Q} | Rsupp(\underline{x}) \subseteq U_i \text{ for some }i\}|\\
                      &=\tilde{Q}^k - |\cup_{i=1}^{\frac{q^n-1}{q-1}} \tilde{\C}(U_i)|.
 \end{align}
 We use $s_i$ to denote the 1 dimensional subspace $U_i^{\perp}$ for $1 \leq i \leq \frac{q^n-1}{q-1}$.
 Note that 
 \begin{align*}
 \dim \tilde{\C}(U_i)&= \dim_{\F_{\tilde{Q}}}{\C} - \rho_\C({U_i}^{\perp})\\
 &= k - \r_{Cl(\mathcal{M}_{\C})}(s_i)\\
 &= \nn_{Cl(\mathcal{M}_{\C})}^*(P(\E)\backslash s_i). 
 \end{align*}
 Since $\tilde{\C}(U_i) \cap \tilde{\C}(U_j) = \tilde{\C}(U_i \cap U_j)$, similarly as above we get,
 \begin{align*}
 \dim~ \tilde{\C}(U_i) \cap \tilde{\C}(U_j) &= \dim_{\F_{\tilde{Q}}}{\tilde{\C}} - \rho_\C({U_i}^{\perp} \cup {U_j}^{\perp})\\
 &= k - \r_{Cl(\mathcal{M}_{\C})}(\{s_i,s_j\})\\
 &= \nn_{Cl(\mathcal{M}_{\C})}^*(P(\E)\backslash \{s_i,s_j\}).
 \end{align*}
Following the same argument, we can say 
\[
\dim~ \tilde{\C}(\cap_{j=1}^{a} U_{i_j}) = \nn_{Cl(\mathcal{M}_{\C})}^*(P(\E)\backslash \{s_{i_1},\ldots, s_{i_a}\}),\] for any positive integer $a$.
Therefore
\begin{align*}
\cup_{i=1}^{\frac{q^n-1}{q-1}} \tilde{\C}(U_i)| &= \sum\limits_{i} |\tilde{\C}(U_i)| - \sum\limits_{i,j} |\tilde{\C}(U_i) \cap \tilde{\C}(U_j)| +\cdots + (-1)^{\frac{q^n-1}{q-1}-1} \sum |\cap_{j=1}^{\frac{q^n-1}{q-1}} \tilde{\C}(U_{i_j})|\\
&=\sum\limits_{i} \tilde{Q}^{\nn_{Cl(\mathcal{M}_{\C})}^*(P(\E)\backslash s_i)} - \sum\limits_{i,j}\tilde{Q}^{\nn_{Cl(\mathcal{M}_{\C})}^*(P(\E)\backslash \{s_i,s_j\})}  +\cdots + \\
&\qquad \qquad(-1)^{\frac{q^n-1}{q-1}-1} \tilde{Q}^{\nn_{Cl(\mathcal{M}_{\C})}^*(P(\E)\backslash P(\E))} \\
&= (-1)^{{\frac{q^n-1}{q-1}-1}}(\sum\limits_{\gamma \subseteq P(\E)}(-1)^{|\gamma|} \tilde{Q}^{\nn_{Cl(\mathcal{M}_{\C})}^*(\gamma)}
\end{align*}
Therefore $A_{\C,E}^{\tilde{Q}} = (-1)^{{\frac{q^n-1}{q-1}}} \sum\limits_{\gamma \subseteq P(\E)}(-1)^{|\gamma|} \tilde{Q}^{\nn_{Cl(\mathcal{M}_{\C})}^*(\gamma)}.$
\end{proof}
\begin{remark}
$A_{\C,E}^{\tilde{Q}}$ is of course non-zero only if $m \ge n$, since otherwise there cannot be any codeword of rank $n$ in $\tilde{\C}$.
\end{remark}

 Proposition \ref{help} can be viewed as a variant of \cite[Formula (10)]{JRV16}, the proof of which was inspired by \cite[Section 5.5.]{JP}.
 
Using the result and procedure above we find expressions for the 
 $A_{\C,s}^{\tilde{Q}}$ for $s=0,1,\cdots,n-1$. For that first we define the following $q$-matroid.
 \begin{definition}
 Let $\C$ be a Gabidulin rank-metric code of length $n$ over $\FQ/\Fq$ and let $U$ be a subspace of $\E=\Fq^n$ with $\dim_{\Fq} U =s$. By considering $\C(U)$ to be a Gabidulin rank-metric code over $\FQ/\Fq$, its corresponding $q$-matroid is defined as $\MM_{\C(U)}:=(U,\rho_U)$, where $U$ is identified with $\Fq^s$ and for any subspace $V \subseteq U$, $$\rho_U(V)=\dim ~\C(U)-\dim~\C(U)(V_U^{\perp}),$$  where $V_U^{\perp}$ is the orthogonal complement of $V$ in $U$ w.r.t. some chosen basis of $U=\Fq^s$.
 \end{definition}
  %Let $\C(U)$ as usual denote the subcode 
 %of $\C$ consisting of those elements in $\C$ whose row space is contained in $U$ (so $\dim~ \C(U)= \dim ~\C -\rho_C(U^{\perp}$).
 %Identifying $U=\Fq^s$ we consider $\C(U)$ as a rank-metric code in its own right. As such it has an associated 
%rank function $\rho_U(V)=\dim ~\C(U)-\dim~\C(U)(V_U^{\perp} )$ and thus it defines a $q$-matroid $\mathcal{M}_U=\MM_{\C(U)}=(U,\rho_U)$. Here $V_U^{\perp} $ is the orthogonal complement of $V$ with respect to a chosen basis for $U=F_q^s$.
\begin{lemma}\label{same}
For a Gabidulin rank-metric code $\C$ over $\FQ/\Fq$ of length $n$ and for any $\Fq$-subspace $U$ of $\Fq^n$, let $\eta_U^*$ be the nullity function of the $q$-matroid $\MM_{\C(U)}^*$. Then $\eta_U^* = \eta_{\C}^*$. In other words, the $q$-cycles of $\mathcal{M_C}^*$ contained in $U$ and the $q$-cycles of $\MM_{\C(U)}^*$ are the same.
\end{lemma}

\begin{proof}
 Since $V \subseteq U$, whatever basis we pick, we obtain $\eta_U^*(V)=\dim_{\F_q}C(U)(V)=\dim_{\F_q}C(V)=\eta_{\C}^*(V)$ and thus the statement of the lemma follows.
\end{proof}
%Whatever basis we pick, we obtain that $\eta_U^*(V)=\dim_{\F_q}C(U)(V)=\dim_{\F_q}C(V)=\eta_{\C}^*(V)$. This implies that the  $q$-cycles of \trygve{$\mathcal{M_C}^*$} contained in $U$ and the $q$-cycles of $\MM_{\C(U)}^*$ are the same (in stark contrast to the statement in Remark \ref{contrast} below). 

\begin{definition} \label{identify}
For any subspace $U \subset \E=\Fq^n$, with a fixed basis and dot product, and any $q$-matroid $\mathcal{M}=(\E,\rho)$ with conullity function $\eta^*$, we let $\mathcal{M}_U=(U,\rho_U)$ be the $q$-matroid with ground space $U$ and conullity function $\eta_U^*$, such that $\eta_U^*(V)=\eta^*(V),$ for all subspaces $V \subset U$. 
\end{definition}
From Lemma \ref{same} it is then clear that $(\mathcal{M}_C)_U=\mathcal{M}_{C(U)}.$

 Using Proposition \ref{help} we then obtain the following expression for $A_{\C,U}^{\tilde{Q}}$:

\begin{proposition} \label{help2}
Let $\C$ be a Gabidulin rank-metric code over $\F_{Q}/\Fq$ of length $n$. Let $\mathcal{M}=\MC = (\Fq^n, \rho_{\C})$ be its corresponding $q$-matroid and $Cl(\MC)=(P(\Fq^n), \r)$ be the associated classical matroid. Then for $\tilde{Q} = Q^r$ and $U \subseteq \Fq^n$,
$$A_{\C,U}^{\tilde{Q}}= (-1)^{{\frac{q^s-1}{q-1}}}\sum\limits_{\gamma \subseteq P(U)} (-1)^{|\gamma|} \tilde{Q}^{{\nn_{Cl(\mathcal{M}_U)}^*}(\gamma)},$$
where $\nn_{Cl(\mathcal{M}_U)}^*$ is the nullity function of the dual classical matroid $Cl(\mathcal{M}_U)^*.$
\end{proposition}

\begin{proof}
Recall that $A_{\C,U}^{\tilde{Q}}$ is the number of codewords in $\C\otimes_{\FQ}\tilde{Q}$, whose rank support is exactly $U$ or the number of codewords in $\tilde{\C}(U)$ whose rank support is exactly $U$. Since $\tilde{\C}(U) = \C(U) \otimes_{\FQ} {\F_{\tilde{Q}}}$, it is clear that $A_{\C,U}^{\tilde{Q}}=A_{\C(U),U}^{\tilde{Q}}$. Thus the result directly follows from Proposition \ref{help}.
\end{proof}
\begin{remark} \label{contrast}
{\rm In stark constrast to the statement in Lemma \ref{same}}, $\nn_{Cl(\mathcal{M}_U)}^*(\gamma)$ is not in general equal to 
$\nn_{Cl(\mathcal{M})}^*(\gamma)$, for $\gamma$ contained in $P(U)$.
We have $\nn_{Cl(\mathcal{M})}^*(\gamma)=0$ for all $\gamma \in P(U)$
if $U \ne E$  (all cycles of $Cl(\mathcal{M})^*$ are to big to be contained in such a $P(U)$).
\end{remark}

Nevertheless, as an immediate consequence of Proposition \ref{help2}, we get:
\begin{corollary} \label{help3}
 	$$A_{\C,s}^{\tilde{Q}} =  (-1)^{{\frac{q^s-1}{q-1}}}
 	\sum\limits_{U, \dim~U=s}\sum\limits_{\gamma \subseteq P(U)} (-1)^{|\gamma|} \tilde{Q}^{{\nn_{Cl(\mathcal{M}_U)}^*}(\gamma)},$$
 	for $s=0,1,\cdots,n-1,n$.
\end{corollary}
Here the sum is over all $U$ with $\dim~U =s$.
Since the number of $\gamma$ and $U$ is finite for a fixed $q$, we conclude that:
\begin{corollary} \label{old}
%For each $s=1,\cdots,n$ and a Gabidulin code $\C$ we have: The %numbers 
%$A_{\C,s}^{\tilde{Q}}$, where $\tilde{Q}=Q^r=q^{mr},$ are %simultaneously  given by a single polynomial in $\tilde{Q}$, and %coefficients in $\mathbb{Z}$, for all  $r\in \mathbb{N}$, and the %degree of this polynomial is at most \rakhi{$\max\{ \dim~\C(U) \colon %U \subseteq \Fq^n \text{ with } \dim~U =s\}$}, which, in any case, is %at most $k=\dim~ \C$.

There exists a polynomial $P \in \mathbb{Z}[X]$ of degree $$\max\{ \dim~\C(U) \colon U \subseteq \Fq^n \text{ with } \dim~U =s,\}$$ such that $P(Q^r)=A^{Q^r}_{C,s}$, for all $r \in \mathbb{Z}.$
\end{corollary}

\begin{proof}
The only thing left to prove is the statement about the degrees. All exponents occurring in the expression for 
$A_{\C,s}^{\tilde{Q}}$ are at most the maximum of the 
numbers  $\nn_{U,Cl(\mathcal{M})}^*(U)$ for all $U$ 
of dimension $s$. But these numbers are simply $\dim~ \C(U)$.
\end{proof}

 \begin{remark}
{\rm Corollary \ref{old} can be easily derived from descriptions by other authors, and then typically from descriptions of Delsarte codes in general,  but we have included it here for completeness of our own exposition. See for example \cite[Remark 3.5 and Theorem 3.8]{BCR}.}
\end{remark}

 \section{Number of codewords of each rank weight and Betti numbers} \label{ncwq}	 
 
We will briefly demonstrate another well known and more direct way to find the 	
$A_{\C,s}^{\tilde{Q}}$,  i.e. the number of words of rank weight $s$ in $\C \otimes_{\F_{Q}} \F_{\tilde{Q}}$ for  $0 \leq s \leq n,$ where $\tilde{Q} = Q^r$. 

%As before, let
%$A(U)=A_{\C,U}^{\tilde{Q}}$ denote the number of  codewords in the same code with support exactly $U$, for each $\Fq$-subspace $U$ of $\Fq^n$, and let $\tilde{\C}(U)$ denote the number of  codewords in the same code with support contained in $U$.
From the definition, it follows that
$$|\tilde{\C}(U)|= \sum_{V \subseteq U}A_{\C,V}^{\tilde{Q}}.$$
M{\"o}bius inversion gives:
$$A_{\C,U}^{\tilde{Q}}=\sum_{V \subseteq U}(-1)^{\dim U - \dim V}q^{{\dim U - \dim V \choose\ 2 }}|\tilde{\C}(V)|.$$
But following Lemma \ref{independent}, we have$$|\tilde{\C}(V)|=\tilde{Q}^{\dim_{F_{\tilde{Q}}}(\tilde{\C}(V))} =\tilde{Q}^{\dim_{\FQ}\C(V)}, \text{ and }$$ 
 $${\dim_{\FQ}\C(V)}=\dim_{\FQ}\C - \rho_{\C}(V^{\perp})=k - \rho_{\C}(V^{\perp})= \dim~U - \rho_{\C}^*(V)=
\eta_{\C}^*(V).$$
This gives:
\begin{proposition} \label{help4}
$$A_{\C,s}^{\tilde{Q}} = \sum_{\substack{U \subseteq \E \\ \dim~ U = s}}\sum_{V \subseteq U}(-1)^{\dim~U - \dim~V}q^{{\dim~U - \dim~V \choose\ 2 }}\tilde{Q}^{\eta_{\C}^*(V)},$$
for $s=1,\cdots,n.$
\end{proposition}
An advantage with this expression, compared with that in Corollary \ref{help3}, is that the conullity $\eta_{\C}^*$ refers to 
the same $q$-matroid $\MC$, for all the $U$ appearing in the formula.
 	 
 	 Comparing the two expressions of $A_{\C,s}^{\tilde{Q}}$ in Corollary \ref{help3} and Proposition \ref{help4}, we have the following result.
 	 \begin{corollary} \label{Zpolynomial}
 	 Let $\mathcal{M}=\MC = (\E= \Fq^n, \rho)$ for a Gabidulin rank-metric code $\C$ and let $Cl(\MM_U)$ be the classical matroid corresponding to the $q$-matroid $\MM_U$ for a subspace $U$ of $\E$. Then
 	 as formal polynomials in $\mathbb{Z}[X]$ we have:
 	\small $$ (-1)^{{\frac{q^s-1}{q-1}}}
 	\sum_{\substack{U \subseteq \E \\ \dim U = s}}\sum\limits_{\gamma \subseteq P(U)} (-1)^{|\gamma|}X^{{\nn_{Cl(\mathcal{M}_U)}^*}(\gamma)}=
 	 \sum_{\substack{U \subseteq \E \\ \dim U = s}}\sum_{V \subseteq U}(-1)^{\dim U - \dim V}q^{{\dim U - \dim V \choose\ 2 }}X^{\eta_{\MM}^*(V)},$$
 	for $s=0,1,\cdots,n-1,n.$
 	 \end{corollary}
 \begin{proof}
 These are both polynomials, and the
 difference between them has zeroes for $Q^r$, for infinitely many $r$. But any non-zero polynomial over any field (in this case $\mathbb{Q}$ or $\mathbb{R}$) has only finitely many zeroes. Hence the difference between the two polynomials appearing in the corollary is the zero polynomial.
 \end{proof}
 \begin{definition}
 For a Gabidulin rank-metric code $\C$, we use $A_{\C,s}(X)$ to denote the polynomial(s) in Corollary \ref{Zpolynomial}. This is called the $s$-th generalized rank weight polynomial of the code $\C$.
 \end{definition}

 \begin{definition}
  For any $q$-matroid $\mathcal{M}$, let $P_{\mathcal{M},s}(X)$ denote the polynomial appearing on the left side in Corollary \ref{Zpolynomial}. We call the $P_{\mathcal{M},s}(X)$ the $s$-th generalized rank weight polynomial of a $q$-matroid $\mathcal{M}$.
  %When the $q$-matroid $\mathcal{M}$ is understood, we denote them %simply by $P_{s}$.
  \end{definition}
  \begin{remark}
 Note that $A_{\C,s}(\tilde{Q})=P_{\MM_{\C},s}(\tilde{Q}) =A^{\tilde{Q}}_{\C,s}$ for $\tilde{Q} = Q^r$ for any $r \in \mathbb{N}$ for any Gabidulin rank-metric code $\C$.
 \end{remark}

%(Recall that $\eta_{Cl(\mathcal{M}_U)}^*$ refers to (the dual of) as ``usual" matroid, 
%namely the (dual of the) "classification" of the matroid associated to the $q$-matroid $\mathcal{M}_U$ 
%(Contraction of $M$ in $U^{\perp}$, I believe, as a generalization of taking the matroid corresponding to $\C(U)$), 
%while $\eta^*$ refers to the (dual of) the $q$-matroid $\MC$ we started with..
 
 %\begin{remark}
 %{\rm We call the $P_{\mathcal{M},s}(X)$ the generalized weight polynomials of a $q$-matroid $\mathcal{M}$, and when the $q$-matroid $\mathcal{M}$ is understood, we denote them \sout{$P_{\MM,s}(X)$ or} simply by $P_{s}$.}
 %\end{remark}
 
% \rakhi{ Combine definition 51 and remark 52}
 %\trygve{DELETE THIS:} \begin{corollary} \label{c1}
%We recall Definition \ref{dr}, valid for any  $q$-matroid %$\mathcal{M}$:  ${d}_i(\mathcal{M})= \min \{\dim~ U| \eta^*(U)=i\}$ %for \textit{any} $q$-matroid $\mathcal{M}$.
%Then ${d}_i(\mathcal{M}) = \min \{s| \deg P_{\mathcal{M},s}(X) = %i\}.$
%   \end{corollary}
   From this result we obtain:
 \begin{corollary} \label{newd}
 For a Gabidulin rank-metric code we have:
 ${d}_i = \min \{s| \deg P_{\MM_{\C},s}(X) = i\}$,
 for $i=1,\cdots,k$, and  $P_{\MM_{\C},s}$ the $s$-th generalized rank weight polynomial of the $q$-matroid $\MC$.
 \end{corollary}
\begin{proof}
From Definition \ref{dr}, valid for any  $q$-matroid $\mathcal{M}$:  ${d}_i(\mathcal{M})= \min \{\dim~ U| \eta^*(U)=i\}$ for \textit{any} $q$-matroid $\mathcal{M}$. Compare with the right side of Corollory \ref{Zpolynomial}.
\end{proof}

 We now have:
 \begin{proposition} \label{upperpolynomial}
 For a $q$-matroid $\mathcal{M} = (\E=\Fq^n, \rho)$, let $N=Cl(\mathcal{M})^*$. Then we have:
  $$P_{\mathcal{M},n}(X)=\sum_{l=0}^{k}\sum_{i=0}^{k}(-1)^i({\beta^{(l)}_{i,\frac{q^n-1}{q-1}}}(N)-\beta^{(l-1)}_{i,\frac{q^n-1}{q-1}}(N))X^l=$$
  
   $$\sum_{l=0}^{k}\sum_{i=0}^{k}(-1)^i({\beta^{(l)}_{i,P(\E)}}(N)-\beta^{(l-1)}_{i,P(\E)}(N))X^l.$$
 \end{proposition}
 \begin{proof}
 This is a special case of \cite[Theorem 5.1]{JRV16} which (in the relevant case) says that for any classical matroid $M=(E,\mathbf{r})$ of rank $k$:
 the polynomial
 $$(-1)^{|E|}\sum_{\gamma \subset E}(-1)^{|\gamma|}
 X^{\nn_{M^*}(\gamma)}.$$
 is equal to
 $$\sum_{l=0}^{k}\sum_{i=0}^{k}(-1)^i({\beta^{(l)}_{i,E}(M^*)-\beta^{(l-1)}_{i,E}}(M^*))X^l.$$
% (where the indexing of the $i$ are shifted by $1$),
Comparing with the left version of $P_{\mathcal{M},n}$ in Corollary \ref{Zpolynomial}, or the expression in Proposition \ref{help2}, one obtains the result.
 \end{proof}
 Let $U$ be a cycle of dimension  $s \in \{1,\cdots,n\}$ for a $q$-matroid $\mathcal{M}$.
 %$\MM_{\C}^*=(E,\rho_{\C}^*)$.
 Then Proposition \ref{upperpolynomial} immediately gives:
 $$P_U(X)= \sum_{l=0}^{k_U}\sum_{i=0}^{k_U}(-1)^i(\beta^{(l)}_{i,P(U)}(N_U)-\beta^{(l-1)}_{i,P(U)}(N_U))X^l,$$ 
 %for the number of codewords of rank support $U$, 
 where $\beta^{(l)}_{i,P(U)}(N_U)$ and $\beta^{(l-1)}_{i,P(U)}(N_U)$ refer to Betti numbers of the classical matroid
 $N_U=Cl(\mathcal{M}_U)^*$, and $k_U=\dim~ M_U$.
 Hence we obtain that $P_{\mathcal{M},s}(X)$ is the sum of all such expressions for all $U$ of dimension $s$. We recall that for a Gabidulin rank-metric code $\C$ then  $P_U(\tilde{Q})$ is the number of codewords of rank support $U$ in $\tilde{\C}$.
 We would like to relate the Betti numbers appearing in 
 these expressions, and which refer to different matroids
 $(Cl(M_{U}))^*$ to Betti numbers of one \textit{single} matroid $Cl(\mathcal{M})^*$ (which $P_{\mathcal{M},n}(X)$) already does, but none of the other $P_{\mathcal{M},s}$ so far). To remedy this lack of simplicity we refer to the Corollary \ref{cor:betti} below, given in \cite[Corollary 2]{JV21}, and using the exposition on p. 59 in  \cite{S77}. But first we recall the following classical definition:
 
 \begin{definition}
Let $L$ be a lattice. The M{\"o}bius function $\mu_L(a,b)$ is defined recursively by $\mu_L(a,a)=1$, and $\mu(a,b)=-\Sigma_{a \le c < b} \mu(a,c)$.
\end{definition}
 \begin{corollary}\cite[Corollary 21]{JV21}\label{cor:betti}
For a matroid $M=(E,\mathbf{r})$ and a subset $X \subset E$ we have \[\beta_{\nn(X),X} = (-1)^{\nn(X)}\mu_{L_F(M^*)}(E\backslash{X},E) = (-1)^{\nn(X)}\mu_{L_C(M)}(\emptyset,X),\]
where $L_F(M^*)$ and $L_C(M)$ refer to the lattices of flats of $M^*$ and cycles of $M$, respectively.
\end{corollary}
 Since we know that $\beta_{i,X}=0$ for all $i$ different from the nullity $\nn(X)$, Corollary \ref{cor:betti} implies that the
Betti numbers of a classical matroid, and also of all of its elongation matroids, are entirely determined by the lattice of cycles of the matroid. This makes the following result important:

\begin{proposition} \label{compare4}
Let $\mathcal{M}$ be a $q$-matroid on $\E=\Fq^n$ and $U$ be a $q$-cycle of the dual $q$-matroid $\mathcal{M}^*$. %$\mathcal{M}_{\C}^*$ for a Gabidulin code $\C$. 
Then the following $4$ lattices are isomorphic
\begin{itemize}
\item The sublattice of $q$-cycles of $\mathcal{M}^*$ contained in  $U$.
\item  The lattice of $q$-cycles of $\mathcal{M}_U^*$.
\item The lattice of cycles of $Cl(\mathcal{M}_{U})^*$.
\item The sublattice of cycles of $Cl(\mathcal{M})^*$ contained in
the cycle $R(U)=P(\E)- P(U^{\perp}).$ 
\end{itemize}
\end{proposition}
\begin{proof}
From Definition \ref{identify}, it follows directly that the two first lattices are identical. The second and third lattices are isomorphic, by Proposition \ref{compare2}(b). Thus it is sufficient to show that the fourth lattice is isomorphic to any of the three lattices above.

Note that, since $U$ is a $q$-cycle of $\mathcal{M}^*$, then $R(U)=P(\E)- P(U^{\perp})$ is a cycle of $Cl(\mathcal{M})^*$. Furthermore, if $V,W$ are $q$-cycles contained in the $q$-cycle $U$, then  $V \subseteq W$ if and only if $P(V) \subseteq P(W)$ if and only if $R(V) \subseteq R(W)$). Hence the first and fourth lattices are isomorphic and we get the desired result.
\end{proof}
By Corollary \ref{cor:betti} and Proposition \ref{compare4} the M{\"o}bius functions of all these lattices can be expressed by Betti numbers relating to the single classical matroid $N=Cl(\mathcal{M}_{\C}^*).$ 
In particular the isomorphism between the third and fourth lattices above gives that 
$$\beta^{(l)}_{i,P(U)}(N_U)=\beta^{(l)}_{i,R(U)}(N|_{R(U)})= \beta^{(l)}_{i,R(U)}(N).$$
%where the left expression refers to the matroid
%$Cl(\mathcal{M})^*|_{R(U}$ which pr. definition is equal to %$Cl(\mathcal{M})^*$ restricted to $R(U)$, and the right expression refers to the matroid $Cl(M_{\C(U)}^*$. Moreover, the left expression,
%which in the outset refers to a classical matroid with %ground set $R(U)$ is the same as when the expression 
%refers to the matroid $Cl(\mathcal{M})^*$ with ground set $P(\E)-P(E^{\perp})=P(\E)$. 
For the rightmost equality we have used that in general $\beta_{i,\sigma}(M)=
\beta_{i,\sigma}(M|_{\sigma})$ for a matroid $M$. We then obtain:
\begin{proposition} \label{submain} Given a $q$-matroid $\mathcal{M}$. Then
 $$P_{\mathcal{M},s}(X)= \sum_{\dim U =s}\sum_{l=0}^k \sum_{i=0}^{k}(-1)^i(\beta^{(l)}_{i,R(U)}(N)-\beta^{(l-1)}_{i,R(U)})(N))X^l,$$
 for the classical matroid $N=Cl(\mathcal{M})^*$.
\end{proposition}
Since we know that there are no other cycles of $Cl(\mathcal{M})^*$ than those of the form $R(U)$, we also know that
$\beta^{(l-1)}_{i,X}=0$ for all $X \subseteq P(\E)$ not of this form.
Hence we obtain
%(recall the notation $[s]=q^{n-1}+q^{n-2}+\cdots+q^{n-s}$):
\begin{theorem} \label{main}
 Let $\mathcal{M}$ be a $q$-matroid. Then the $s$-th generalized weight polynomial of $\mathcal{M}$ is given by $$P_{\mathcal{M},s}(X)=\sum_{l=0}^{k} \sum_{i=0}^{k}(-1)^i(\beta^{(l)}_{i,[s]}(N)-\beta^{(l-1)}_{i,[s]}(N))X^l,$$
 where $N=Cl(\mathcal{M})^*$, the dual classical matroid corresponding to $\mathcal{M}$. Consequently, the number of codewords of rank weight $s$ in $\tilde{\C}$ is   $$A_{\C,s}(\tilde{Q})=\sum_{l=0}^{k} \sum_{i=0}^{k}(-1)^i(\beta^{(l)}_{i,[s]}(N)-\beta^{(l-1)}_{i,[s]}(N))\tilde{Q}^l.$$
 
\end{theorem}
%\begin{remark}
%The upper limits for $i$ in the results starting with %Proposition \ref{upperpolynomial} could all be replaced by 
%$\infty$, or by $k$, which is the highest a priori %possible value of the nullity for any cycle appearing.
%\end{remark}
 
We recall the definition: 
 ${d}_i(\mathcal{M})= \min \{\dim~ U| \eta^*(U)=i\})$,
 valid for all $q$-matroids $\mathcal{M}$.
 From Proposition \ref{compare4} we also obtain the following generalization of Corollary \ref{Zpolynomial}:
 \begin{corollary} \label{c1}
For any  $q$-matroid $\mathcal{M}$ we have:  ${d}_i(\mathcal{M}) = \min \{s| \deg P_{\mathcal{M},s}(X) = i\}.$
\end{corollary}
\begin{proof}
We recall the definition:
$$ P_{\mathcal{M},s}(X)=(-1)^{{\frac{q^s-1}{q-1}}}
 	\sum_{\substack{U \subseteq \E \\ \dim U = s}}\sum\limits_{\gamma \subseteq P(U)} (-1)^{|\gamma|}X^{{\nn_{Cl(\mathcal{M}_U)}^*}(\gamma)}.$$
 	But by Proposition \ref{compare4} (the isomorphism between the first and the third lattice there), even the third equality below holds: 
 	\begin{align*}
 	   &\min \{s| \deg P_{\mathcal{M},s}(X) = i\}\\
 	   &=\min \{s|\text{ there exist } \gamma,U \text{ with }\dim(U)=s \text{ and }  \gamma \subset P(U) \text{ and }  \nn_{Cl(\mathcal{M}_U)}^*(\gamma)=i\}\\
 	   &=\min \{s|\text{ there exists }U \text{ with }\dim(U)=s \text{ and } \nn_{Cl(\mathcal{M}_U)}^*(P(U))=i\}\\
 	   &=\min \{s|\text{ there exists }U \text{ with }\dim(U)=s\text{ and } \eta^*_{\mathcal{M}}(U)=i\}\\
 	   &=d_i.
 	\end{align*}
 	This proves the corollary. 
\end{proof} 
 \begin{definition}
  Given a Gabidulin rank-metric code $\mathcal{C}$. Then the $i$-th generalized rank weight distribution of $\C$ 
 %$\C \otimes_{\F_{Q}} \F_{Q^r}$ 
 is the integer vector whose $u$-th component, $0 \leq u \leq n$, is defined by 
% \[A_{\tilde{Q},u}^{(i)} 
 \[A_{\C,u}^{(i)}:= \sum_{\substack{U \subseteq \E \\ \dim~ U = u}} %A_{\tilde{Q},U}^{(i)}, 
 A_{\C,U}^{(i)},
 \]
 where, for any $U \subseteq \E$,
 \[
% A_{\tilde{Q},U}^{(i)}
A_{\C,U}^{(i)}:=|\{\mathcal{D} \subseteq 
%\C \otimes_{\F_{Q}} \F_{Q^r}
\C: \dim(\mathcal{D}) = i, Rsupp(\mathcal{D})=U \}|.
 \] 
 \end{definition}	 
   
  % \begin{proposition}
  % For any $0 \leq u,w \leq n$,
   %\[
  % A_{\tilde{Q},u}^{(i)} = \sum\limits_{w=0}^{u}{{n-w}\brack {u-w}}_q(-1)^{u-w}q^{u-w \choose 2} \sum_{\substack{W \subseteq \E\\ \dim W =w}}{{\eta^*(W)} \brack {i}}_q.
%   \]
   %\end{proposition}
%\begin{proposition}
% Let $\C$ be a $[n,k]_{q^m}$ vector rank-metric code or %Gabidulin code and $\tilde{\C} = \C \otimes_{\F_{q^m}} \F_{q^{mr}}$ be the extended code. Let $\C^r$ be the linear subspace consisting of $r \times n$ matrices over $\F_{q^m}$ whose rows are elements in $\C$. Then there is an isomorphism of $\F_{q^m}$-linear spaces $\C^r$ and $\tilde{\C}$.
%\end{proposition}   
The following two results can be viewed as an adaptation to the rank-metric situation of the arguments given in \cite[Lemma 5.4, Prop. 5.28]{JP13}.   
   \begin{lemma}
   For an element $\mathbf{c} \in \tilde{\C}$, let $M$ be the corresponding $(r \times n)$ matrix, referring to our fixed basis $\{g_1,\cdots,g_r\}$ of $\mathbb{F}_{Q^r}$ over $\mathbb{F}_{Q}.$  Let $D$ be the subspace of $\C$ generated by the rows of the matrix $M$. Then $wt(\mathbf{c}) = wt(D):= \dim_{\Fq} Rsupp(D)$, where $wt(\mathbf{c})$ denotes the rank weight of $\mathbf{c}$.
   \end{lemma}
   \begin{proof}
  % Let $\{1, \alpha^1, \ldots, \alpha^{r-1}\}$ be a basis of $\F_{q^{mr}}$ over $\F_{q^{m}}$ and $\{1, \beta^1, \ldots, \beta^{m-1}\}$ be a basis of $\F_{q^{m}}$ over $\F_{q}$.
    %Therefore, $\{1, \alpha, \ldots, \alpha^{r-1}, \beta, \alpha \beta, \ldots, \alpha^{r-1}\beta, \ldots, \beta^{m-1}, \ldots, \beta^{m-1} \alpha^{r-1}\}$ is an ordered basis of $\F_{q^{mr}}$ over $\F_{q^m}$. 
    Let $\mathbf{c}=(c_1,\ldots, c_n) \in \tilde{\C},$ and $c_j = \sum\limits_{i=0}^{r-1}c_{i,j}g_i$, so that the $ij$th entry of $\A$ is $c_{i,j}$. From definitions of rank weight and rank support, it follows that $wt(\mathbf{c}) = wt(D)$.
    \end{proof}
   \begin{proposition}  \label{thetwo}
   Let $\C$ be a $[n,k]_{q^m}$ Gabidulin rank-metric code. 
   %Suppose $\tilde{\C}$ denotes the extended rank-metric code $\C \otimes_{\F_{q^m}} \F_{q^{mr}}$. 
   Then 
   \[
   A_{\C,w}(q^{mr}) = \sum\limits_{s = 0}^{k} [r,s]_{q^m} A^{(s)}_{\C,w},
   \]
   where $[r,s]_{q^m}$ is the number of $F_{q^m}$-linear subspaces of dimension $s$ contained in $F_{q^m}^r.$
   \end{proposition}

   \begin{proof}
Here $A_{\C,w}(q^{mr})$ is the number of codewords of $\tilde{\C}=\C \otimes_{\F_{q^m}} \F_{q^{mr}}$ of rank weight $w$, which we get by substituting $T = q^{mr}$ in the polynomial $A_{\C,w}(T)$. Now we do the counting in another way. Let $\mathbf{c}$ be an element of $\tilde{\C}$ which corresponds to a $(r \times n)$-matrix $\A$ with rows in $\C$. Let $D$ be the subcode of $\C$ generated by the rows of $\A$ and it has rank weight $w$ and dimension $s$. On the other hand, for any subcode $D_1 \subseteq \C$ of dimension $s$ and rank weight $w$, we consider $\A_1$ to be a generator matrix of $D_1$. Then left multiplication of a $r \times s$ matrix of rank $s$ with $\A_1$ gives an element of $\C^r$, which has the same rank weight $w$. The number of $r \times s$ matrices in $\F_{q^m}$ of rank $s$ is equal to $[r,s]_{q^m} = \Pi_{i=0}^{s-1} ({q^m}^r - {q^m}^i)$. Therefore, the number of codewords of $\tilde{\C}$ of rank weight $w$ is equal $\sum\limits_{s = 0}^{k}[r,s]_{q^m}A^{(s)}_w$.
   \end{proof}

\begin{definition} \label{phi}
Given a $q$-matroid $\MM$.
For each $j \in \{0,1,\cdots,\frac{q^n-1}{q-1}\}$ and $l \in \{0,1,\cdots,k\}$ we set
$$\phi^{(l)}_j=\sum_{i=0}^{k} (-1)^i \beta^{(l)}_{i,j},$$
referring to Betti numbers of $Cl(\mathcal{M})^*$ and its elongations.
We set $\phi_j=\phi^{(0)}_j$.
\end{definition}
\begin{theorem} \label{equivalence}
The following $4$ sets of data are equivalent
for a Gabidulin rank-metric code $\C$ with the associated $q$-matroid $\MM$: 
\begin{itemize}
\item The $s$-th generalized weight polynomial $P_{\mathcal{M},s}(X)$ of $\mathcal{M}$ for all $s$,
\item The $s$-th generalized weight polynomial $A_{\C,s}(X)$ of $\C$ for all $s$,
\item The $j$-th generalized weight distribution ($A_{\C,s}^{(j)}$) of $\C$ for all $j$,
\item The alternative sum of (elongated) Betti numbers $\phi_{j}^{(l)}$ of  $Cl(\mathcal{M})^*$ for all $j,l$.  
\end{itemize}
\end{theorem}
\begin{proof}
By definition, the two polynomials $P_{\mathcal{M},s}(X)$ and $A_{\C,s}(X)$ of $\C$ are same. The equivalence between the $A_{\C,s}(X)$  and the $A_{\C,s}^{(j)}$ is given by Proposition \ref{thetwo}. By Theorem \ref{main} the $P_{\mathcal{M},s}(X)$ are determined by the $\phi_{j}^{(l)}$.
Moreover, Theorem \ref{main} shows that one, starting with $l=0$, can determine the $\phi_{j}^{(l)}$ recursively for all $l$ if one knows the $P_{\mathcal{M},s}(X)$.
\end{proof}
We also obtain:
\begin{corollary} \label{follows}
For a $q$-matroid $\MM$ in general the following are equivalent:
\begin{itemize}
\item The $P_{\mathcal{M},s}(X)$ for all $s$.
%(and thereby the  $A_{\C,s}(Q^r)$ for all $r,s$).
\item The $\phi_{j}^{(l)}$ for all $j,l$.  
\end{itemize}
\end{corollary}
\begin{proof}
In the part of the proof of Theorem \ref{equivalence} which is relevant here, we use Theorem \ref{main}, which is valid for all $q$-matroids.
\end{proof}
\begin{remark}\label{pure}
{\rm One is sometimes interested in $q$-matroids $\mathcal{M}$, such that all $q$-flats of the same rank have the same dimension, for all fixed ranks. These are called $q$-perfect matroid designs ($q$-PMD), e.g.,
see \cite{BCIJS}. This is the same as $Cl(\mathcal{M})$ having all flats of a given fixed rank the same cardinality, and the same as $N=Cl(\mathcal{M})^*$ 
having all its cycles of each fixed nullity the same cardinality. This is equivalent to saying that in the resolution of the Stanley-Reisner ring of $N$, we have $\beta_{i,j} \ne 0$ for only if one $j$ (say $j=d_i$) for each $i=1,\cdots,rk(\mathcal{M})$ (and then $\phi_{d_i}=\beta_{i,d_i}$). Hence the resolution is pure in this sense if and only all $q$-flats of the same rank of $\mathcal{M}$ have the same dimension. In this case all Stanley-Reisner rings of the elongations of $N$ also have pure resolutions, and it is particularly easy to find all $\phi^{(l)}_j$ in terms of the $d_i$,
due to the Herzog-K\"uhl equations, see, for example, \cite{HK} or \cite[Def. 3.1 and Remark 3.2]{AK}}
\end{remark}
\begin{remark}
{\rm \begin{itemize}
\item Theorem \ref{equivalence} is a $q$-analogue of \cite[Corollary 17]{JV20}, which applies to Hamming codes
and associated matroids. In both \cite{JV20} and \cite{JVfirst} one found all the $\phi_{j}^{(l)}$, and using this corollary, one found all the weight spectra, i.e. all the $A_s^{(j)}$ for two kinds of Veronese codes. It is unclear whether such techniques are useful for Gabidulin rank-metric codes, as it is for Veronese Hamming codes.
\item Corollary \ref{follows} indicates an extended range of applications for the reasoning above, including an extension of  Theorem \ref{equivalence}. One could imagine $q$-matroids coming from a wider class of objects than that of Gabidulin rank-metric codes. An example could be \textit{any} subset $\C$ of the space of $(m \times n)$-matrices over $\mathbb{F}^q$, such that the subset $\C(U)$ had cardinality $q^{ms}=Q^s$ for some $s=s(U)$ for all $\mathbb{F}_q$-subspaces $U$ of $\mathbb{F}_q^n$. And $\tilde{\C}$ could be defined as just $\C^r$. The rank function $log_Q|\C|-log_Q|\C(U^{\perp})|$ could then be used to give results like Corollary  \ref{follows} and possibly to Theorem \ref{equivalence} also, for such ``almost affine Gabidulin rank-metric codes", as one could call them. 
To obtain a full extension of Theorem \ref{equivalence} to such codes, one must also then define, and successfully treat, some hierarchy of natural ``almost affine Gabidulin subcodes" from which one could define the $A_s^{(j)}$. Generalizations from linear Hamming codes to almost affine codes were treated in \cite{AS} and \cite{JVAA}. It is not clear to us how interesting it will be to extend the class of  linear (usual), rank-metric  Gabidulin rank-metric codes to such an analogous, bigger class of codes.
\item For classical matroids the $\phi_j$ play a role as certain coefficients of the two-variable coboundary polynomials, as is shown in \cite[Proposition 5]{JV21}. It is conceivable that they may play a similar role for $q$-matroids.
\end{itemize} }
\end{remark}

   \section{ Two different ways of determining rank-weight spectra} \label{examples}
   In this section, we demonstrate with concrete examples how to determine (generalized) rank-weight spectra of Gabidulin rank-metric codes. While in the first example we use the expression in Theorem \ref{main} to determine the rank-weight distribution and Proposition \ref{thetwo} to determine the higher weight spectra, in the second example we consider the class of MRD codes and determine the weight spectra directly from the corresponding (uniform) $q$-matroids.

\begin{example} \label{second}
{\rm  Consider the field extension $\mathbb{F}_{2^4}/\mathbb{F}_{2}$ of degree $4$, and let $a$ be a root in $\mathbb{F}_{2^4}$ of the irreducible polynomial $X^4 + X+ 1$ in $\mathbb{F}_{2}[X]$ so that $\mathbb{F}_{2^4} = \mathbb{F}_{2}(a)$. 
Let $\C$ be the rank-metric code of length 4 over the extension $\mathbb{F}_{2^4}$ of $\mathbb{F}_{2}$ such that a generator matrix of $\C$ is given by 
\[ 
G:= %\left(\begin{array}{cccc}
\begin{pmatrix}
a^{2} + a + 1 \ & a^{2} & a^{3} + a + 1 \ & a^{3} + a^{2} + a + 1 \\
a^{2} + a + 1 & a^{3} + 1 \ & a & a + 1 \\
a^{2} + 1 & 1 & a^{2} + 1 & a^{3} + 1
\end{pmatrix}.
%\end{array}\right) 
\]
Let $\MC$ be the $q$-matroid over $\F_2^4$ corresponding to the rank-metric code $\C$ and let $N$ be the classical matroid $Cl(\MC)^*$. We consider the simplicial complex $\Delta_N$ associated to the matroid $N$. 
%  ??? It is clear that $\dim~\Delta_C = \mathop{\rm rank}(G)=3$. 

We remark that this is a case, where $m=n(=4),$ but we are not making any considerations here, about what happens if we look at column spaces instead of row spaces. We just follow the recipe described above and find the weight spectrum of the code, as described, using row spaces as supports.
Here we list the (elongated) Betti numbers of the Stanley-Reisner ring associated to the simplicial complex $\Delta_{\C}$.
   
\begin{align*}
\beta_{0,0}&=1&\quad &\quad &\quad\\
\beta_{1,8}&=1&\quad &\quad &\quad\\
\beta_{1,12}&=28&\quad &\quad &\quad\\
\beta_{2,14}&=76 &\beta_{1,14}^{(1)}=15&\quad&\quad\\
\beta_{3,15}&=48 ,&\beta_{2,15}^{(1)}=14 &\quad&\beta_{1,15}^{(2)}=1
\end{align*}
  The Betti numbers here are computed using  Magma (\cite{M}), and a program code where one finds the numbers via a certain adaptation of Hochster's formula (\cite{H}), which here may be used to express these numbers in terms of relevant homological invariants.
  
   We use $\beta_{i,[j]}$ to denote the Betti number $\beta_{i, {q^{n-1}+ q^{n-2} + \cdots + q^{n-j}}}$. By Theorem \ref{betti}, the minimum weight of $\C$ is $1$ as $\min\{j \;|\; \beta_{1,[j]} \neq 0\} = 1$. Now we determine the weight spectrum $(A_0, A_1, A_2, A_3, A_4)$ of $\C$ by substituting the values of the Betti numbers in the expression of  $A_{\C,s}(q^m)$ as proved in Theorem \ref{main}. 
   
   For ease of calculation we expand the expression for  $A_{\C,s}(q^m)$ as follows.
   \begin{align*}
    A_s(q^m) &= (-\beta_{1,[s]}^{(0)} + \beta_{2,[s]}^{(0)} - \beta_{3,[s]}^{(0)})+q^m(-\beta_{1,[s]}^{(1)} + \beta_{2,[s]}^{(1)} - \beta_{3,[s]}^{(1)})-q^m(-\beta_{1,[s]}^{(0)} + \beta_{2,[s]}^{(0)} - \beta_{3,[s]}^{(0)})\\
    &\quad \quad +q^{2m}(-\beta_{1,[s]}^{(2)} + \beta_{2,[s]}^{(2)} - \beta_{3,[s]}^{(2)}) -q^{2m}(-\beta_{1,[s]}^{(1)} + \beta_{2,[s]}^{(1)} - \beta_{3,[s]}^{(1)}).
    \end{align*}

    $A_0(2^4)= 1,\; A_1(2^4)= 15,\; A_2(2^4)= 420,\; A_3(2^4)= 2460,\; \text{and} \;A_4(2^4) = 1200$.
    
    }
    %$ where the sum is over the cycles $\sigma$ with $|\sigma|=j.$ }
    \end{example}
     %\begin{align*}
     %A_1(2^4)&= 15,\\
     %A_2(2^4)&= 420,\\
     %A_3(2^4)&= 2460,\\  
     %A_4(2^4) &= -\beta_{3,[4]} + 2^4(\beta_{2,[4]}^{(1)}) - 2^4 (- \beta_{3,[4]}) + 2^{8}(- \beta_{1,[4]}^{(2)}) - 2^8(\beta_{2,[4]}^{(1)}) + 2^12 (\beta_{1,15}^{(2)})\\
    %&=-48 + 16*14 + 16*48 -256*1 - 256*14 + 4096\\
    %&=1200
    %\end{align*}   
   
 \begin{example} \label{first}
 {\rm Let $\C$ be an $[n,k,d]$ MRD code over $\Fqm = \FQ$ and let $\tilde{\C}$ be its extended code $\tilde{\C}= \C \otimes_{\FQ} \FQt$. Let the $q$-matroid associated to the MRD code $\C$ is the uniform $q$-matroid $\mathcal{M}$. 
   
    Step 1. 
    
    Recall that the $s$-th generalized rank-weight polynomial of $\mathcal{M}$ is 
    \[P_{\mathcal{M},s}(X) = \sum\limits_{\dim~ U =s} \sum_{l=0}^k\sum\limits_{i=0}^k (-1)^i (\beta^{(l)}_{i,R(U)}(N)-\beta^{(l-1)}_{i,R(U)})(N))X^l.\]
    Note that for a subspace $U\subseteq \E=\Fq^n$, we have $|R(U)|= |P(\E) - P(U^{\perp})|=q^{n-1} + q^{n-2} + \cdots + q^{n-s}$. Therefore,
    \begin{align*}
        P_{\mathcal{M},d}(X) &= \beta_{1,[d]}^{(0)} X - \beta_{1,[d]}^{(0)},\\
        P_{\mathcal{M},d+1}(X) &=\beta_{1,[d+1]}^{(1)} X^2 - (\beta_{1,[d+1]}^{(1)}+\beta_{2,[d+1]}^{(0)})X + \beta_{2,[d+1]}^{(0)},\\
         P_{\mathcal{M},d+2}(X) &=\beta_{1,[d+2]}^{(2)} X^3 - (\beta_{1,[d+2]}^{(2)}+\beta_{2,[d+2]}^{(1)})X^2 + (\beta_{2,[d+2]}^{(1)}+\beta_{3,[d+2]}^{(0)})X-\beta_{3,[d+2]}^{(0)},\\
          &\vdots \quad \quad \quad \quad \vdots\\
          P_{\mathcal{M},n}(X) &=\beta_{1,[n]}^{(k-1)} X^k - (\beta_{1,[n]}^{(k-1)}+\beta_{2,[n]}^{(k-2)})X^{k-1} +\cdots +(-1)^{k-2}(\beta_{k-1,[n]}^{(1)} +\beta_{k,[n]}^{(0)})X\\ &\quad \quad \quad +(-1)^{k-1} \beta_{k,[n]}^{(0)}.
    \end{align*}
    Step 2. Now we compute the (elongated) Betti numbers of $Cl(\mathcal{M})^*$, the dual of the classical matroid associated to the $q$-matroid.
    
    Consider the lattice of cycles of $Cl(\mathcal{M})^*$, say, $L^*$. We use the following formula for computing Betti numbers: 
    $\beta_{\mathbf{n}^*(X),X}= |\mu_{L^*}(\phi,X)|,$ where $\mathbf{n}^*$ is the nullity function of $Cl(\mathcal{M})^*$.
    
    Note that, if $\mathcal{M}$ is a $q$-matroid of rank $k$, then the cycles of $Cl(\mathcal{M})^*$ with nullity $i$ have  cardinalities $q^{n-1}+ q^{n-2}+ \cdots + q^{k-i}$ for $1 \leq i \leq k$. Let $c_i$ denote a cycle of cardinality $q^{n-1}+ q^{n-2}+ \cdots + q^{k-i}$ with nullity $i$. Note that there are ${{n} \brack {k-i}}$ cycles $c_i$ of nullity $i$ in the lattice $L^*$ and there are ${{k-j} \brack {k-i}} = {{k-j} \brack {i-j}}$ many cycles $c_j$ of nullity $j$ contained in $c_i$. 
    
    \begin{align}
    \text{Then }h_i:=\mu(\emptyset, c_i) &= - \mu(\emptyset, \emptyset) - \sum\limits_{c_j \subsetneq c_i} \mu(\emptyset, c_j)\\
                        &= -1 - {{k-1} \brack {i-1}}\mu(\emptyset, c_{1}) - \cdots - {{k-i+1} \brack {1}}\mu(\emptyset, c_{i-1})\\
                        &= -1 -{{k-1} \brack {i-1}}h_1 -\cdots - {{k-i+1} \brack {1}} h_{i-1}.
    \end{align}
     Here $d=n-k+1$. As we know that $\beta_{i,\sigma} \neq 0 $ if and only if $\sigma \in N_i$ (the cycles of nullity $i$), by Corollary \ref{closingin}, the nonzero $\mathbb{N}^{P(\E)}$-graded Betti numbers of $Cl(\mathcal{M})^*$ are of the form $\beta_{i,{q^{n-1}+ \cdots + q^{k-i}}}$ for $1 \leq i \leq k$. Thus using the recursive formula for the M\"obius function of $L^*$, we get the expression for the non-zero Betti numbers as follows, 
    \[\beta_{i,[n-k+i]} = {{n} \brack {n-k+i}}_q h_i.\]
    
    %Now we derive a recursive formula for $\beta_{i,[j]}$'s. 
    
    \begin{align*}
        \beta_{1,{[n-k+1]}} &= {{n} \brack {n-k+1}}_q|\mu_{L^*}(\emptyset,X)|, \text{ where } X \text{ is a cycle of cardinality }[n-k+1]. \\
        \beta_{2, {[n-k+2]}} &= {{n} \brack {n-k+2}}_q|\mu_{L^*}(\emptyset,X)|, \text{ where } X \text{ is a cycle of cardinality }[n-k+2] \\
                             &= {{n} \brack {n-k+2}}_q |\sum\limits_{X \subsetneq Y} \mu_{L^*}(\emptyset,X) + \mu(\emptyset, \emptyset)|  \\
                             &= {{n} \brack {n-k+2}}_q ({{d+1} \brack {1}}_q -1).\\
        \beta_{i, {[n-k+i]}}&= {{n} \brack {n-k+i}}_q|\mu_{L^*}(\emptyset,X)|, \text{ where } X \text{ is a cycle of cardinality }[n-k+i]. \\
                             \end{align*}
    
     We determine the $l$-th elongated Betti numbers of $Cl(\mathcal{M})^*$ for a fixed $0 \leq l \leq k$. 

    \begin{align*}
      &\beta^{(l)}_{1,c_1} = 1  &\beta^{(l)}_{1,[d+l]}= {{n} \brack {d+l}}_q \beta^{(l)}_{1,c_1}\\
      &\beta^{(l)}_{2,c_2} = {{d+l+1}\brack {1}}_q -1   &\beta^{(l)}_{2,[d+l+1]}= {{n} \brack {d+l+1}}_q \beta^{(l)}_{2,c_2}\\
      &\beta^{(l)}_{3,c_3} ={{d+l+2} \brack {1}}_q \beta^{(l)}_{2,c_2} - {{d+l+2} \brack {2}}_q \beta^{(l)}_{1,c_1}  &\beta^{(l)}_{3,[d+l+2]}= {{n} \brack {d+l+2}}_q \beta^{(l)}_{3,c_3}\\
       \quad&\vdots \quad \quad &\vdots \quad \quad\quad\\
       \end{align*}
       $\beta^{(l)}_{k-l,[n]} = {{n} \brack {n}}_q ({{n} \brack {1}}_q \beta^{(l)}_{k-l-1,c_{k-l-1}} - {{n} \brack {2}}_q \beta^{(l)}_{k-l-2,c_{k-l-2}} +  \cdots + (-1)^{n-d-l} {{n} \brack {n-l}}_q \beta^{(l)}_{1,c_1}).$
    Step 3. 
    
    We determine the weight spectra $(A_{\C,w})$ from the weight enumerator polynomials by determining these values for some particular values of $w$. 
      Since $ A_{\C,w}(Q) = P_{\mathcal{M},w}(Q) $, thus 
     \begin{align*}
       A_{\C,d} = P_{\mathcal{M},d}(X)|_{X = {q^m}} &= \beta_{1,[d]}^{(0)} q^m - \beta_{1,[d]}^{(0)}\\
                    &=  {{n} \brack {d}}(q^m -1).
     \end{align*}
     \begin{align*}
       A_{\C,d+1}= P_{\mathcal{M},d+1}(X)|_{X = {q^m}} &= \beta_{1,[d+1]}^{(1)} q^{2m} -( \beta_{1,[d+1]}^{(1)}+\beta_{2,[d+1]}^{(0)})q^m+\beta_{2,[d+1]}^{(0)}\\
                    &=  {{n} \brack {d+1}}(q^{2m} -{{d+1} \brack {1}}q^m+{{d+1} \brack {1}}-1).
     \end{align*}
   and so on. 
    The weight spectra of MRD codes is determined in \cite{BCR} as follows
    
    \[A_{\C,r} = {{n} \brack {r}} \sum\limits_{i=0}^{r-d} (-1)^i q^{i \choose 2}{{r}\brack {i}}(q^{mk - m(n+i-r)} -1).\]
    For $r=d,d+1$, this gives us 
    \begin{align*}
    A_{\C,d} &= {{n} \brack {d}}(q^{mk - m(n-d)} -1)={{n} \brack {d}} (q^m -1).
    \end{align*} 
     \begin{align*}
      A_{\C,d+1}=   {{n} \brack {d+1}}(q^{2m} -{{d+1} \brack {1}}q^m+{{d+1} \brack {1}}-1)
     \end{align*}
     as expected.
     
    Step 4. 
 The higher weight spectra $\{A_w^{(s)}\}$, for $s\ge2$ can be easily obtained by combining Proposition \ref{thetwo} with the knowledge of the $A_{\C,w}(Q).$
 %(regardless of the way the $A_{\C,w}(Q) $ are obtained, and they can also be found in other, more direct ways in this particular example, which is included hereto demonstrate the method, rather than finding new results).
 { For example we determine the value $A_{d+1}^{(2)}$. 
 
 \begin{align*}
 A_{d+1}(q^{2m})&=\sum\limits_{s = 0}^{k}[r,s]_{q^m}A^{(s)}_{d+1}\\
                &=\sum\limits_{s = 0}^{k}\Pi_{i=0}^{s-1} ({q^m}^r - {q^m}^i) A^{(s)}_{d+1}\\ 
                &= (q^{2m} - 1) A^{(1)}_{d+1} + (q^{2m} - 1)(q^{2m}-q^m) A^{(2)}_{d+1}\\
                &= (q^m+1)A^{(0)}_{d+1} + (q^{2m} - 1)(q^{2m}-q^m) A^{(2)}_{d+1}\\
                &=(q^m+1){{n} \brack {d+1}}(q^{2m} -{{d+1} \brack {1}}q^m+{{d+1} \brack {1}}-1)\\
                &\qquad \qquad \qquad +(q^{2m} - 1)(q^{2m}-q^m) A^{(2)}_{d+1}\\
 \end{align*}
 Since $A_{d+1}(q^{2m}) = {{n} \brack {d+1}}(q^{4m} -{{d+1} \brack {1}}q^{2m}+{{d+1} \brack {1}}-1)$, after extending and simplifying the above equation we get,
 
 %&{{n} \brack {d+1}}(q^{4m} -{{d+1} \brack {1}}q^{2m}+{{d+1} \brack {1}}-1)- (q^m+1){{n} \brack {d+1}}(q^{2m} -{{d+1} \brack {1}}q^m+{{d+1} \brack {1}}-1)\\
 \begin{align*}
 {{n} \brack {d+1}} (q^{4m} - q^{3m} -q^{2m} +  q^m) &= (q^{2m} - 1)(q^{2m}-q^m) A^{(2)}_{d+1}\\\
 {{n} \brack {d+1}} (q^{2m} -1)(q^{2m} -  q^m) &= (q^{2m} - 1)(q^{2m}-q^m) A^{(2)}_{d+1}\\
  A^{(2)}_{d+1} &= {{n} \brack {d+1}}.
 \end{align*}}
 }
 
 \end{example}

    \section{Virtual Betti numbers} \label{mbs}
    The two examples in Section \ref{examples} are really different, in the sense that in Example \ref{second} we work with the Stanley-Reisner ring of the classical matroid $N=Cl(\mathcal{M})^*,$ find its independence complex, and cycles, and treat it as any classical matroid, regardless of the fact that it ``comes from" a $q$-matroid. In Example \ref{first}, however, we do not touch any classical matroid at all, in our concrete computations.
    The classical matroid basically just serves as a justification there, to work with M{\"o}bius functions of the lattice of $q$-cycles of $\mathcal{M}^*$. 
We only know how to associate Stanley-Reisner rings to the independence complex of classical matroids and their elongations.
From those rings we have seen that we can derive the Betti numbers of minimal resolutions as modules over polynomial rings.
Moreover, it is clear that, via Corollary \ref{cor:betti}, we can associate well-defined M{\"o}bius-numbers to each such Betti number, where these numbers are defined in terms of the lattice of cycles of the matroid in question. Equivalently, it can be defined in terms of the inverted lattice, in this case the geometric lattice of flats of the dual matroid. We now define:
\begin{definition}
{\rm Let $L$ be a lattice with rank function $r$. For $l \in \{0,1,\cdots,r(L)\},$ let $L^{(l)}$ be the lattice obtained by replacing all the points of $L$ of rank at most $l$ by a single point, which then becomes the zero of $L^{(l)}$. Then the rank function of $L^{(l)}$ is given by $r^{(l)}(P)=r(P)-l$ for any $P \in L$ with $r(P) \geq l$ and $0$, otherwise. 
%Given any point $P$ in a lattice $L$, 
 %and $l \in \{0,1,\cdots,r(L)\},$ let \rakhi{$L^{(l)}$} be the lattice obtained by replacing all points of rank at most $l$ by a single point, which then becomes the zero of $L^{(l)}$. and $r$ is the original rank function of $L$, while the $0$ in $\mu^{(l)}(0,P)$ is that of $L^{(l)}$. Clearly its rank function is $r^{(l)}(P)=r(P)-l$. 
 We set
\[
V^{(l)}_{i,P}= \begin{cases}(-1)^{r(P)-l}\mu^{(l)}(0,P) &\text{ if }r(P)=l+i, \\
    0 &\text{ otherwise},
 \end{cases} \text{ and }V_{i,P}=V^{(0)}_{i,P}.
 \]
 Here
 $\mu^{(l)}$ is the M{\"o}bius function of the lattice $L^{(l)}$. }
%Moreover we set:
%$$V_{i,P}=V^{(0)}_{i,P}.$$
\end{definition}
%Here
 %$\mu^{(l)}$ is the M{\"o}bius function of the lattice $L^{(l)}$ obtained by replacing all points of rank at most $l$ by a single point, which then becomes the zero of $L^{(l)}$, and $r$ is the original rank function of $L$, while the $0$ in $\mu^{(l)}(0,P)$ is that of $L^{(l)}$. Clearly its rank function is $r^{(l)}(P)=r(P)-l$. 
 
 \begin{definition}
 {\rm Let $L$ be a lattice with rank function $r$. Then for a given non-negative function $f:L \rightarrow\ \mathbb{N}_{0},$ we set 
$$V_{i,j}^{(l),f}=\sum_{r(P)=l+i,f(P)=j}V^{(l)}_{i,P}\textrm{, and  } V_{i,j}^f=V_{i,j}^{(0),f}\textrm{, and }V_{i}^f=\sum_jV_{i,j}^f.$$ 
% Here the sum is only over those $P$, such that $f(P)=j$.}
We call $V_{i,j}^{(l),f}$ the $l$-th elongated \textit{virtual $\mathbb{N}$-graded Betti numbers} and $V_{i}^f$ the \textit{ungraded virtual Betti numbers}.}
\end{definition} 
For a Gabidulin rank-metric code $\C$, we have already seen that the lattice of $q$-cycles of $\mathcal{M}_{\C}^*$ and the lattice of cycles of $N^*=Cl(\mathcal{M}_{\C})^*$ are isomorphic and if $U$ be a $q$-cycle with $dim~ U=j$, then the cardinality of its corresponding cycle $R(U)$ of $N^*$ is $q^{n-1}+q^{n-2}+\cdots+q^{n-j}$.
We then associate the functions $f_1$ and $f_2$ to these lattices, respectively, where $f_1(U)=\dim~ U$, and $f_2(X)=|X|$ if $X$ is 
of the form $R(U)$ for $U$ a $q$-cycle of dimension $j$, and $0$ otherwise (or just $f_2=|X|$ for all $X \subseteq P(\E)$). It is then evident that:
$$\beta^{(l)}_{i,R(U)}=V_{i,U}^{(l),f_1},$$ for each $q$-cycle $U$, and that:
$$\beta^{(l)}_{i,[j]}=V_{i,j}^{(l),f_1}.$$

So in Example \ref{second}, we really did use the ``virtual Betti numbers" $V_{i,j}^{(l)}=V_{i,j}^{(l),f_1}$
to determine $A_{\C,w}(Q)$ (or $A_{\C,w}(\tilde{Q})$ if one prefers). It is clear in that all results in the previous sections, involving the $\beta_{i,R(U)}^{(l)}$ or the 
$\beta^{(l)}_{[j]}$ one can replace these invariants with the virtual Betti numbers, and thus reformulate all these results, only referring directly to the lattice of $q$-cycles of $\mathcal{M}_{\C}^*$. In practical situations, one might prefer to use different lattices depending on the tools one has at hand, as explained in Examples \ref{first} and \ref{second}.

$\text{                                                                                                }$

\it{On behalf of all authors, the corresponding author states that there is no conflict of interest.}
%\end{remark}

\end{document}